\pdfoutput=1

\documentclass[options]{article}
\usepackage{amsmath}
\usepackage{amsfonts}
\usepackage{amssymb}
\usepackage{amsthm}
\usepackage{esint}
\usepackage{mathtools}
\usepackage{titling}
\usepackage{authblk}
\providecommand{\keywords}[1]{\textbf{\textit{Keywords ---}} #1}
\usepackage{commath}
\usepackage{xcolor}
\usepackage{graphicx}
\usepackage{graphbox}
\usepackage[toc,title,page]{appendix}
\usepackage{subcaption}
\usepackage{bbm}
\usepackage{hyperref}
\usepackage{color,soul}
\usepackage{bm}
\usepackage{mathrsfs}

\usepackage{enumitem}
\usepackage[margin=1in]{geometry}

\newtheorem{theorem}{Theorem}
\newtheorem{corollary}{Corollary}
\newtheorem{lemma}{Lemma}
\newtheorem{remark}{Remark}

\newtheorem{proposition}{Proposition}
\newtheorem*{remark*}{Remark}
\numberwithin{equation}{section}
\newcommand\omicron{o}

\title{On an Anisotropic Fractional Stefan-Type Problem \\with Dirichlet Boundary Conditions}
\author{Catharine W.K. Lo\thanks{CMAFcIO -- Departamento de Matem\'atica, Faculdade de Ci\^encias, Universidade de Lisboa P-1749-016 Lisboa, Portugal\\ Email address: cwklo@fc.ul.pt} \, and Jos\'e Francisco Rodrigues\thanks{CMAFcIO -- Departamento de Matem\'atica, Faculdade de Ci\^encias, Universidade de Lisboa P-1749-016 Lisboa, Portugal\\ Email address: jfrodrigues@ciencias.ulisboa.pt}}
\date{}

\begin{document}
\maketitle

\begin{abstract}
In this work, we consider the fractional Stefan-type problem in a Lipschitz bounded domain $\Omega\subset\mathbb{R}^d$ with time-dependent Dirichlet boundary condition for the temperature $\vartheta=\vartheta(x,t)$, $\vartheta=g$ on $\Omega^c\times]0,T[$, and initial condition $\eta_0$ for the enthalpy $\eta=\eta(x,t)$, given in $\Omega\times]0,T[$ by \[\frac{\partial \eta}{\partial t} +\mathcal{L}_A^s \vartheta= f\quad\text{ with }\eta\in \beta(\vartheta),\] where $\mathcal{L}_A^s$ is an anisotropic fractional operator defined in the distributional sense by \[\langle\mathcal{L}_A^su,v\rangle=\int_{\mathbb{R}^d}AD^su\cdot D^sv\,dx,\] $\beta$ is a maximal monotone graph, $A(x)$ is a symmetric, strictly elliptic and uniformly bounded matrix, and $D^s$ is the distributional Riesz fractional gradient for $0<s<1$. We show the existence of a unique weak solution with its corresponding weak regularity. We also consider the convergence as $s\nearrow 1$ towards the classical local problem, the asymptotic behaviour as $t\to\infty$, and the convergence of the two-phase Stefan-type problem to the one-phase Stefan-type problem by varying the maximal monotone graph $\beta$.
\end{abstract}


\section{Introduction}

The classical Stefan problem, in an open bounded Lipschitz domain $\Omega \ni x=(x_1,\dots,x_d)$ and for time $t\in [0,T]$, can be formulated in $Q_T=\Omega\times]0,T[$ by an evolution equation involving a subdifferential operator \begin{equation}\label{ClassicalStefan}\frac{\partial }{\partial t}\beta(\vartheta)-D\cdot (AD \vartheta) \ni f,\end{equation} where $\vartheta(x,t)$ is the temperature, $D$ is the gradient, $A=A(x)$ is a symmetric, strictly elliptic and bounded matrix, and $\beta$ corresponds to a maximal monotone graph, such that $\beta(r)=b(r)+
\lambda\chi$ for $\chi\in H(\vartheta)$ for the maximal monotone graph $H(r)$ associated with the Heaviside function, i.e. $H(r)=0$ for $r<0$, $H(r)=1$ for $r>0$, $H(0)=[0,1]$, and $b$ a given continuous and strictly increasing function, $\lambda>0$ (see Figure \ref{fig:EnthalpyGraph}) with inverse $\gamma=\beta^{-1}$ satisfying $\lim_{r\to+\infty}\gamma(r)=+\infty$ and $\lim_{r\to-\infty}\gamma(r)=-\infty$ for the two-phase problem and $\gamma(r)=0$ for $r\leq \lambda$ for the one-phase problem. The notation $\beta(\vartheta)$ should be understood as follows: there exists a section $\eta$ of the multifunction $\beta(\vartheta)$ which satisfies the required conditions. In turn, $\vartheta$ is easy to recover from $\eta$ since $\beta^{-1}=\gamma$ is a single-valued mapping. For works on the variational formulation of the classical Stefan problem, see for instance \cite{OleinikStefanPaper}, \cite{KamenomostskajaStefan}, \cite{FriedmanStefanPbTAMS}, Chapter V.9 of \cite{LadyzhenskayaSolonnikovUraltsevaBook}, Section 3.3 of \cite{LionsBookNonlinear}, \cite{Damlamian1977}, \cite{TarziaThesis}, \cite{RodriguesStefanRevisited}, \cite{RodriguesVarMethods1994b} and \cite{VisintinStefan}. 

We can also consider the one-phase problem (I) as the limit of the two-phase problem (II). Indeed, physically, for large Stefan number, the liquid phase only contributes exponentially small terms to the location of the solid–melt interface. Therefore, at times close to complete solidification, the temperature in the liquid essentially vanishes and the two-phase problem reduces to the one-phase problem. For more detailed discussions, see \cite{McCueWuHill20082ph1ph}. See also \cite{BarbaraStoth19972ph1ph} for the one-dimensional case in the classical setting $s=1$.

Here, we consider the corresponding fractional Stefan-type problem, given in $Q_T$ by \begin{equation}\label{FracStefanTheta}\frac{\partial }{\partial t}\beta(\vartheta)+\mathcal{L}_A^s \vartheta\ni f,\end{equation} where $\mathcal{L}_A^s=-D^s\cdot AD^s$ is a non-local operator defined with the distributional Riesz fractional derivatives, with anisotropy given by a measurable matrix $A=A(x)$, which is symmetric, strictly uniformly elliptic and bounded independent of time satisfying \begin{equation}\label{Acoer}a_* |z|^2\leq A(x)z\cdot z\leq a^* |z|^2\end{equation} for almost every $x\in\mathbb{R}^d$ and all $z\in\mathbb{R}^d$. Then, the classical problem \eqref{ClassicalStefan} corresponds to the case $s=1$, i.e. \eqref{FracStefanTheta} with the operator $\mathcal{L}_A^1$, where $D^1=D$.

\begin{figure}[tb]
    \centering
    \includegraphics[width=.45\textwidth,align=c]{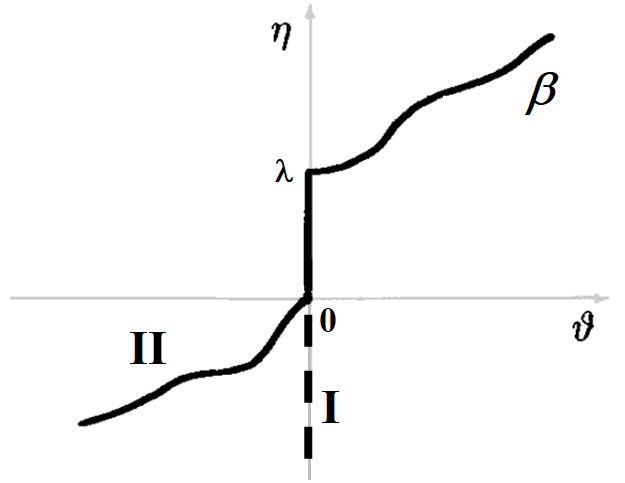}
    \includegraphics[width=.45\textwidth,align=c]{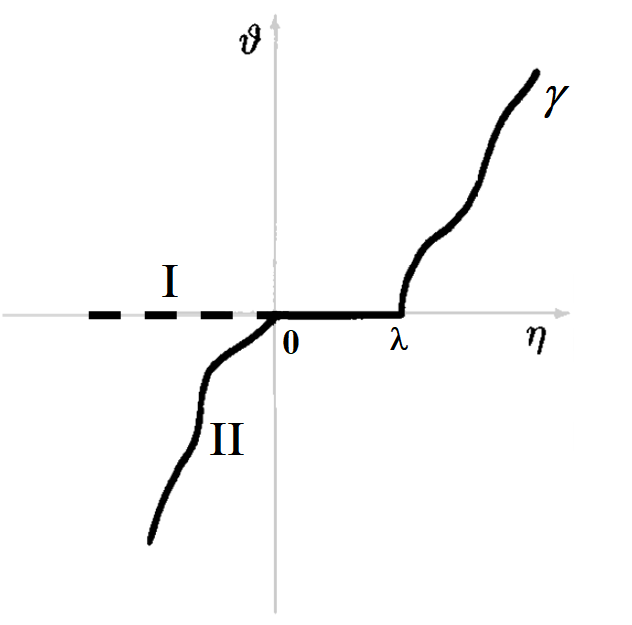}
    \caption{The maximal monotone graphs $\beta$ as the inverse of the continuous monotone functions $\gamma=\beta^{-1}$ in the case of the II phases and the I phase Stefan problems}
    \label{fig:EnthalpyGraph}
\end{figure}

The operator $\mathcal{L}_A^s$ can be viewed as an anisotopic generalisation of the fractional Laplacian. Indeed, following the works of Silhavy \cite{Silhavy}, Shieh-Spector \cite{SS1}--\cite{SS2} and Comi-Stefani \cite{Comi2}--\cite{Comi1}, the Riesz fractional $s$-gradient ($D^s$) and the $s$-divergence ($D^s\cdot$) are defined in integral form for sufficiently regular functions $u$ and vector fields $\bm{\phi}$, respectively, by \[D^su(x):=c_{d,s}\int_{\mathbb{R}^d}\frac{u(x)-u(y)}{|x-y|^{d+s}}\frac{x-y}{|x-y|}\,dy\] and \[D^s\cdot\bm{\phi}(x):=c_{d,s}\int_{\mathbb{R}^d}\frac{\bm{\phi}(x)-\bm{\phi}(y)}{|x-y|^{d+s}}\cdot\frac{x-y}{|x-y|}\,dy,\] where $c_{d,s}=2^s\pi^{-\frac{d}{2}}\frac{\Gamma\left(\frac{d+s+1}{2}\right)}{\Gamma\left(\frac{1-s}{2}\right)}$. Then, in the distributional sense, it is well-known that \[-D^s\cdot D^su=(-\Delta)^su\] for $u\in C_c^\infty(\Omega)$ (see for instance, \cite{SS1}, \cite{Silhavy}), where  $(-\Delta)^s$ is the fractional Laplacian defined as  \[(-\Delta)^su(x)=2^{2s}\pi^{-\frac{d}{2}}\frac{\Gamma\left(\frac{d+2s}{2}\right)}{\Gamma(-s)}\,P.V.\int_{\mathbb{R}^d}\frac{u(y)-u(x)}{|x-y|^{d+2s}}\,dy\quad\text{ for }0<s<1.\] Furthermore, we have the convergence of the fractional derivatives to the classical derivatives as $s\nearrow1$, i.e. \[D^su\to Du,\] as in Comi-Stefani \cite{Comi2}, Bellido et al. \cite{BellidoCuetoMoraCorral2021CVPDEgammaconvg} and Lo-Rodrigues \cite{FracObsRiesz}, for $u\in H^1(\mathbb{R}^d)$ and $u\in H^1_0(\Omega)$.

In this work, we are concerned with the classical fractional Sobolev space $H^s_0(\Omega)$ in a bounded domain $\Omega\subset\mathbb{R}^d$ with Lipschitz boundary, for $0<s<1$, defined as \[H^s_0(\Omega):=\overline{C_c^\infty(\Omega)}^{\norm{\cdot}_{H^s}},\] with \begin{equation}\label{Hs0GraphNorm}\norm{u}_{H^s}^2=\norm{u}_{L^2(\mathbb{R}^d)}^2+\norm{D^su}_{L^2(\mathbb{R}^d)^d}^2,\end{equation} where $u$ is extended by 0 in $\mathbb{R}^d\backslash\Omega$, so that this extension is also in $H^s(\mathbb{R}^d)$. By the classical fractional Poincar\'e inequality (see Lemma \ref{Poincare} below), we shall consider the space $H^s_0(\Omega)$ with the following equivalent norm  \begin{equation}\label{Hs0EquivNorm}\norm{u}_{H^s_0(\Omega)}^2=\norm{D^su}_{L^2(\mathbb{R}^d)^d}^2.\end{equation}

We subsequently denote the dual space of $H^s_0(\Omega)$ by $H^{-s}(\Omega)$ for $0<s\leq1$. Then, by the Sobolev-Poincar\'e inequalities, we have the compact embeddings \[H^s_0(\Omega)\hookrightarrow L^q(\Omega), \quad\  L^{q'}(\Omega)\hookrightarrow H^{-s}(\Omega)=(H^s_0(\Omega))'\] for $1\leq q<2^*$, where $2^*=\frac{2d}{d-2s}$ and $q'>2^\#=\frac{2d}{d+2s}$ when $s<\frac{d}{2}$, and if $d=1$, $2^*=q$ for any finite $q$ and $2^\#=\frac{q}{q-1}$ when $s=\frac{1}{2}$ and $2^*=\infty$ and $2^\#=1$ when $s>\frac{1}{2}$. We recall that those embeddings are continuous also for $q=2^*$ when $s<\frac{d}{2}$ (see for example, Theorem 4.54 of \cite{Demengel}). 

The nonlocal operator $\mathcal{L}_A^s=-D^s\cdot AD^s$ may be defined in the duality sense for $u\in H^s(\mathbb{R}^d)$: \begin{equation}\label{LAdef}\langle\mathcal{L}_A^su,v\rangle:=\int_{\mathbb{R}^d} AD^su\cdot D^sv \quad\forall v\in H^s_0(\Omega),\end{equation} with $v$ extended by zero outside $\Omega$, defining an operator from $H^s(\mathbb{R}^d)$ to $H^{-s}(\Omega)$ since $AD^su\in L^2(\mathbb{R}^d)^d$. Also for $u\in H^s_0(\Omega)$, since we can extend it by 0 outside $\Omega$ to obtain a function in $H^s(\mathbb{R}^d)$, $\mathcal{L}_A^s:H^s_0(\Omega)\to H^{-s}(\Omega)$ can also be represented by \begin{equation}\label{LAdefNotDist}\mathcal{L}_A^su=-D^s\cdot(AD^su).\end{equation}

Given any $\tilde{g}\in H^s(\mathbb{R}^d)$, we introduce $g\in H^s(\mathbb{R}^d)$ defined on the whole space $\mathbb{R}^d$ which satisfies $g|_{\Omega^c}=\tilde{g}$ and is $\mathcal{L}^s_A$-harmonic in $\Omega$, that is to say, we solve the Dirichlet problem with  $g=\tilde{g}\text{ a.e. on }\Omega^c$ for the equation \begin{equation}\label{DirichletBdryCond}\mathcal{L}_A^sg=0\text{ in }H^{-s}(\Omega)\end{equation} in a weak sense, which means
\[\int_{\mathbb{R}^d}AD^sg\cdot D^sv=0\quad\forall v\in H^s_0(\Omega).\]
Note that this is possible and defines $g$ a.e. in $\mathbb{R}^d$ by Lax-Milgram theorem (see Appendix \ref{Sect:DirBdryCondAp}, and also Theorem 1.13 of \cite{SS1}), since $A$ is strictly elliptic and bounded. 

Next, we introduce the enthalpy function \begin{equation}\label{enthalpy}\eta(x,t)\in\beta(\vartheta(x,t))\text{ for almost every }(x,t)\in Q_T\end{equation} with initial condition \begin{equation}\label{enthalpyinitial}\eta(0)=\eta_0\text{ in }H^{-s}(\Omega),\end{equation} and we prescribe a Dirichlet boundary condition \begin{equation}\label{tempdircond}\vartheta(t)=\tilde{g}(t)\text{ a.e. in }\Omega^c=\mathbb{R}^d \setminus\Omega,\text{ for a.e. }t\in]0,T[,\end{equation} for a given $\tilde{g}(t)\in H^s(\mathbb{R}^d)$. For simplicity we shall often describe this Dirichlet condition by saying that $\vartheta(t)-\tilde{g}(t)\in  H^s_0(\Omega)$ for a.e. $t$, which is certainly clear for $s>1/2$, by the trace theorem, and an abuse of notation for $s\leq 1/2$.  Now, for almost every $t\in[0,T]$, introducing $g(t)=\tilde{g}(t)$ in $\Omega^c$  and such that $\mathcal{L}_A^sg(t)=0$ in $\Omega$ in the distributional sense, assuming $f\in L^2(0,T;H^{-s}(\Omega))$, we then have the following weak formulation of the Stefan-type problem when viewed as a single-unknown problem: \begin{equation}\label{FracStefanEta}\left\langle\frac{d\eta}{dt},\xi\right\rangle+\langle \mathcal{L}^s_A(\gamma(\eta)-g),\xi\rangle=\langle f,\xi\rangle,\quad\forall \xi\in L^2(0,T;H^s_0(\Omega))\end{equation} with initial data \eqref{enthalpyinitial}, where $\langle \cdot,\cdot\rangle$ denotes the duality between $L^2(0,T;H^{-s}(\Omega))$ and $L^2(0,T;H^s_0(\Omega))$. Here the Lipschitz graph $\gamma$, which may have flat parts, is defined as the inverse of the maximal monotone graph $\beta$ (see Figure \ref{fig:EnthalpyGraph}). We call the solution $\eta$ of \eqref{FracStefanEta} the \emph{generalised solution} for the enthalpy formulation, by requiring
\[\eta\in H^1(0,T;H^{-s}(\Omega))\cap L^2(Q_T)\text{ with }\gamma(\eta)-g\in L^2(0,T;H^s_0(\Omega)).\]

By the regularity of $\eta$, setting $\beta=b+\lambda H$, we can write $\eta=[b(\vartheta)+\lambda\chi]\in\beta(\vartheta)$ with $\chi\in H(\vartheta)$ a.e. in $Q_T$, i.e. \[0\leq\chi_{\{\vartheta>0\}}\leq\chi\leq 1-\chi_{\{\vartheta<0\}}\leq1\quad\text{ a.e. in }Q_T.\] Suppose we take a more regular test function $\xi$ which additionally satisfies $\xi(T)=0$. Then, using integration by parts in time, we also have a weak variational formulation, with $f\in L^2(0,T;L^2(\Omega))$ and $\eta_0\in L^2(\Omega)$, for the solution $\vartheta=\gamma(\eta)$, i.e. $\vartheta$ is the \emph{weak solution} for the temperature formulation: \begin{equation}\label{SolnSpace}(\vartheta,\chi)\in [L^2(Q_T)]^2,\chi\in H(\vartheta)
\text{ and } \vartheta-g\in  L^2(0,T;H^s_0(\Omega)) \end{equation} satisfy
\begin{equation}\label{FracStefan}-\int_{Q_T}[b(\vartheta)+\lambda\chi]\frac{\partial \xi}{\partial t}+\int_{\mathbb{R}^d\times[0,T]}AD^s\vartheta\cdot D^s\xi=\int_{Q_T}f\xi+\int_\Omega \eta_0\xi(0),\quad\forall \xi\in\Xi^s_T,\end{equation} where \[\Xi^s_T:=\{\xi\in L^2(0,T;H^s_0(\Omega))\cap H^1(0,T;L^2(\Omega)):\xi(T)=0\text{ in }\Omega\}.\]
Compare with \cite{OleinikStefanPaper}, \cite{KamenomostskajaStefan} and Section V.9 of \cite{LadyzhenskayaSolonnikovUraltsevaBook} for the classical case with $s=1$.

\begin{remark}
Note that the variational problem \eqref{FracStefanEta} incorporates the Dirichlet condition \eqref{tempdircond} in the original problem given in \eqref{FracStefanTheta} because of the definition \eqref{DirichletBdryCond}. Since this implies $\int_{\mathbb{R}^d\times[0,T]}AD^sg\cdot D^s\xi=0$ for all $\xi\in\Xi^s_T$, we obtain \eqref{FracStefan} without that term.

Although in general $\eta,\vartheta$ may be nonzero outside $\Omega$, except for the bilinear form $\int_{\mathbb{R}^d\times[0,T]}AD^s\vartheta\cdot D^s\xi$, the other integral terms in the variational formulation \eqref{FracStefan} are only integrated over $\Omega$ in space, since the test function $\xi$ is 0 in $\Omega^c\times]0,T[$.
\end{remark}

Different non-local versions of Stefan-type problems have previously been considered, including in \cite{MR3023404} and \cite{MR3129012} for nonsingular integral kernels, in  \cite{VOLLER20105622}, \cite{MR3400331}, \cite{Ryszewska} and \cite{roscani2021onephase} for the fractional Caputo derivatives, and in \cite{MR3892428}, \cite{MR4013925}, \cite{DelTesoVazquez2Phase}, \cite{DelTesoVazquez1Phase} and \cite{ishige2021refined} for the fractional Laplacian and its nonlocal integral generalization in \cite{athanasopoulos2021twophase}. Stefan-type problems that are fractional in the time derivative have also been considered (see, for instance,  \cite{Tarzia2Phase}, \cite{MR3241498} and \cite{MR3657877}.)

Indeed, when the matrix $A$ is a multiple of the identity matrix, the fractional Stefan-type problem \eqref{FracStefanTheta} reduces to that with the fractional Laplacian as considered in \cite{MR3892428}--\cite{DelTesoVazquez1Phase}. Furthermore, in instances as described in Section 2.3 of \cite{FracObsRiesz} when the fractional operator $\mathcal{L}^s_A$ is replaced with a nonlocal operator $\tilde{\mathcal{L}}^s_a$, corresponding to a Dirichlet form with the kernel $a$ which satisfies some compatibility conditions, \eqref{FracStefanTheta} may also be considered a nonlocal Stefan problem, as considered in \cite{athanasopoulos2021twophase}. However, an equivalence relation between the fractional operator with the matrix $A$ and the nonlocal operator with the kernel $a$ cannot be established in general except in the isotropic homogeneous case (for more details, see Section 2.3 of \cite{FracObsRiesz}), so the two Stefan-type problems with those two operators are not equivalent.

In this paper, we show the existence of a unique solution for the fractional Stefan-type problem with Dirichlet boundary conditions, where the spatial operator is a general anisotropic non-local singular operator of fractional type as given by \eqref{LAdef}, and we keep the classical temperature-enthalpy relation illustrated in Figure \ref{fig:EnthalpyGraph}. This relation in the classical equation \eqref{ClassicalStefan} incorporates, in a generalised form, the free boundary condition relating the balance between the normal velocity of the interface and the jump of the local anisotropic heat flow. In 1-dimension, the extension of the classical free boundary Stefan condition to fractional diffusion, as in the recent paper \cite{Ryszewska} with the fractional Caputo derivative in the nonlocal diffusive term, can be easily made explicit. Similar explicit formulation can be used with the 1-dimensional fractional Riesz spatial derivative when, for each fixed time, the free boundary is a point.

However, in higher dimensions, the Riesz fractional $s$-gradient, as proposed in \cite{Silhavy}, is an appropriate fractional operator maintaining translational and rotational invariance, as well as homogeneity of degree $s$ under isotropic scaling, and so the $\mathcal{L}^s_A$ operator gives a natural and appropriate  anisotropic generalisation of the fractional Laplacian. Keeping the generalised Stefan condition in the evolution equation \eqref{FracStefanTheta} involving the maximal monotone operator $\beta$ is a natural generalisation for the formulation of the anisotropic Stefan problem, extending \cite{DelTesoVazquez2Phase} and \cite{DelTesoVazquez1Phase}, which corresponds to the case where the matrix $A$ is the identity matrix in the unbounded domain. Such an anisotropic operator is coordinate invariant, which makes it more suitable in higher dimensions. Furthermore, the use of this $\mathcal{L}^s_A$ operator allows us to recover the classical Stefan problem when $s=1$, which is in accordance with a requirement of weak continuity from the nonlocal model to the local model, when $s\nearrow1$. However, a main issue remains open in the fractional multidimensional model, namely what is the physical meaning of the Stefan condition due to the lack of a convenient interpretation and definition for the fractional heat flux across the solid-liquid interface.

In Sections \ref{Sect:EnthalpyExist} and \ref{Sect:TempExist}, we employ Hilbertian techniques to show the existence of a generalised enthalpy solution and a weak temperature solution to the initial and boundary value two-phase Stefan-type problem \eqref{FracStefanEta}--\eqref{FracStefan} following the approach of Damlamian \cite{damlamian1976thesis}--\cite{Damlamian1977} for the classical case $s=1$. 

Making use of convergence properties of the fractional derivatives to the classical derivatives when $s\nearrow1$, we show, in Section \ref{Sect:sConvg}, that the solution of the fractional Stefan-type problem converges to the solution of the classical case corresponding to $s=1$. Next, we consider the asymptotic behaviour of the solution as $t\to\infty$ in Section \ref{Sect:AsympBehavT}. Such convergence properties apply to both the two-phase problem, and the one-phase problem, which corresponds to the case of a nonnegative temperature.
The one-phase problem (I) is recovered in Section \ref{Sect:OnePhase} from the two-phase problem (II), when the maximal monotone graph for (II) (see Figure \ref{fig:EnthalpyGraph}) degenerates to that of the one-phase problem (I).

Finally, we complete our paper with three appendices: \ref{Sect:DirBdryCondAp} on the time dependent Dirichlet problem for the fractional operator $\mathcal{L}_A^s$, \ref{Sect:VarIneqEquiv} on the variational inequality formulation for the two-phase and the one-phase problems, and \ref{Sect:Eigen} on the stability of the eigenvalues and eigenfunctions for the operator $\mathcal{L}_A^s$ in $H^s_0(\Omega)$ with respect to $s$ including the convergence $s\nearrow1$.

\section{Existence of the Generalised Enthalpy Solution $\eta$}\label{Sect:EnthalpyExist}

Let $\mathcal{L}=\mathcal{L}^s_A$ be the duality mapping defined by \begin{equation}\label{Fdef}\langle \mathcal{L}u,v\rangle =\int_{\mathbb{R}^d}AD^su \cdot D^sv=:[u,v]_A=(U,V),\end{equation} from $H^s_0(\Omega)$ to $H^{-s}(\Omega)$ with $H^s_0(\Omega)$ identified to a subspace of $L^2(\Omega)$.  Here $\langle \cdot,\cdot\rangle$ is the duality between $H^{-s}(\Omega)$ with $H^s_0(\Omega)$, with $u,v$ extended by zero outside $\Omega$. The equality of the inner product in $H^{-s}(\Omega)$ given by $(\cdot,\cdot)$, with the topology endowed from $\mathcal{L}$, with the equivalent inner product $[\cdot,\cdot]_A$ in $H^s_0(\Omega)$ holds by Riesz representation theorem, with \begin{equation}\label{RieszRep}U=\mathcal{L}u\text{ and }V=\mathcal{L}v\text{ respectively}.\end{equation} This is possible by assumption \eqref{Acoer} and the Poincar\'e inequality, as long as $\Omega$ is bounded. 

In this section, we consider the two-phase problem with \begin{equation}\label{gammacond}\text{ $\gamma$ is Lipschitz with Lipschitz constant $C_\gamma$ such that $\gamma(0)=0$ }\quad\text{ and }\quad \liminf_{|r|\to+\infty}\frac{\gamma(r)}{r}>0.\end{equation}
We prove an existence theorem for the enthalpy $\eta$ similar to the classical case, as given in \cite{Damlamian1977} and \cite{damlamian1976thesis} (See also \cite{VisintinStefan} for further developments). 
To do so, we need a result of Attouch-Damlamian \cite{attouch1975problemes}--\cite{AttouchDamlamian2} in the case where the Hilbert space $H$ is $H^{-s}(\Omega)$.
\begin{proposition}\label{AttouchDamlamianThm}[Theorem 1 of \cite{attouch1975problemes}, and \cite{AttouchDamlamian2}]
Let $(\varphi_t)_{t\in[0,T]}$ be a family of lower semi-continuous convex functions on a Hilbert space $H$. Assume that there exists a function $\mathfrak{a}$ belonging to $BV(0,T)$ such that the following holds: \begin{equation}\label{AttouchDamlamianThmHypo}\varphi_t(V)\leq\varphi_\tau(V)+|\mathfrak{a}(t)-\mathfrak{a}(\tau)|(\varphi_\tau(V)+|V|+1),\quad\forall 0\leq \tau\leq t\leq T,\forall V\in H.\end{equation} Then, for $U_0\in D(\varphi_0)=\{U_0\in H:\varphi_0(U_0)<+\infty\}$ and $F\in L^2(0,T;H)$, there is a unique solution $U \in H^1(0,T;H)$ satisfying \begin{equation}\label{AttouchDalamianThmCauchyEq}\frac{dU}{dt}+\partial\varphi_t(U)=F,\quad U(0)=U_0.\end{equation} Furthermore, the following estimates hold independent of $\varphi$: \begin{equation}\label{AttouchDalamianThmEst1}\norm{U}_{C([0,T];H)}\leq C_1\left(\norm{U_0}_{H},\norm{F}_{L^1(0,T;H)},\norm{\mathfrak{a}}_{BV}\right),\end{equation}\begin{equation}\label{AttouchDalamianThmEst2}\norm{\frac{dU}{dt}}_{L^2(0,T;H)}\leq C_2\left(\norm{U_0}_{H},\varphi_0(U_0),\norm{F}_{L^2(0,T;H)},\norm{\mathfrak{a}}_{BV}\right),\end{equation}\begin{equation}\label{AttouchDalamianThmEst3}\norm{\varphi_t(U)}_{L^\infty(0,T)}\leq C_3\left(\norm{U_0}_{H},\varphi_0(U_0),\norm{F}_{L^2(0,T;H)},\norm{\mathfrak{a}}_{BV}\right).\end{equation}
\end{proposition}

Making use of this proposition, we can show the following existence result.

\begin{theorem}\label{VarStefanThm1}
Let $f\in L^2(0,T;H^{-s}(\Omega))$ and $\tilde{g}\in BV(0,T;L^2(\Omega))\cap L^2(0,T;H^s(\mathbb{R}^d))$, and define $g$ as in \eqref{DirichletBdryCond}, so $g$ satisfies the same regularity as $\tilde{g}$ (see Appendix \ref{Sect:DirBdryCondAp}). Assume $\eta_0\in L^2(\Omega)$ and $\gamma$ satisfies \eqref{gammacond}. Then there exists a unique generalised enthalpy solution $\eta$ to the problem \eqref{FracStefanEta} with initial condition \eqref{enthalpyinitial}, such that \begin{equation}\label{EnthalpyReg}\eta\in L^\infty(0,T;L^2(\Omega))\cap H^1(0,T;H^{-s}(\Omega))\end{equation} and it satisfies \begin{equation}\label{Estimate1}\norm{\eta}_{C([0,T];H^{-s}(\Omega))}\leq C_1\left(\norm{f}_{L^1(0,T;H^{-s}(\Omega))},\norm{\eta_0}_{H^{-s}(\Omega)},\norm{g}_{BV(0,T;L^2(\Omega))}\right),\end{equation}\begin{equation}\label{Estimate2}\norm{\frac{d\eta}{d t}}_{L^2(0,T;H^{-s}(\Omega))}\leq C_2\left(\norm{f}_{L^2(0,T;H^{-s}(\Omega))},\norm{\eta_0}_{L^2(\Omega)},\norm{g}_{BV(0,T;L^2(\Omega))}\right),\end{equation}
\begin{equation}\label{Estimate3}\norm{\eta}_{L^\infty(0,T;L^2(\Omega))}\leq C_4\left(\norm{f}_{L^2(0,T;H^{-s}(\Omega))},\norm{\eta_0}_{L^2(\Omega)},\norm{g}_{BV(0,T;L^2(\Omega))}\right),\end{equation} where $C_1,C_2$ are exactly the constants from \eqref{AttouchDalamianThmEst1}--\eqref{AttouchDalamianThmEst2}, while $C_4$ depends on \eqref{AttouchDalamianThmEst3} and \eqref{AttouchDamlamianThmHypo}.
Furthermore, the corresponding weak temperature solution $\vartheta=\gamma(\eta)$ satisfies
\begin{equation}\vartheta-g\in L^2(0,T;H^s_0(\Omega))\end{equation} and, in addition, it solves \eqref{FracStefan} when $f\in L^2(0,T;L^2(\Omega))$.
\end{theorem}

\begin{proof}
We apply Proposition \ref{AttouchDamlamianThm} with $|\mathfrak{a}(t)-\mathfrak{a}(\tau)|=\norm{g(t)-g(\tau)}_{L^2(\Omega)}$ to the following functions $\phi_t$ on the Hilbert space $H^{-s}(\Omega)$ given for each $t\in[0,T]$ by \begin{equation}\label{ConvexFunctional}\phi_t(W)=\begin{cases}\int_\Omega(j(W)-g(t)W)\,dx&\text{ for }W\in L^2(\Omega);\\+\infty &\text{ for }W\in H^{-s}(\Omega)\backslash L^2(\Omega)\end{cases}\end{equation} where $j$ is the primitive of $\gamma$ such that $j(0)=0$. Then, $j$ is quadratic and the domain $D(\phi_t)$ of $\phi_t$ is given by \begin{equation}\label{ConvexFunctionalDomain}D(\phi_t)=\{W\in H^{-s}(\Omega):\phi_t(W)<\infty\}=L^2(\Omega)\end{equation} thanks to the Cauchy-Schwarz inequality and making use of the fact that $W$ lies in $L^2(\Omega)$. 
It is well-known (see for instance, Theorem 17 of \cite{BrezisMMPDE}) that $\phi_t$ is lower semi-continuous, convex, proper and coercive on $H^{-s}(\Omega)$.  Furthermore, there exist constants $\delta$ and $c$ such that \begin{equation}\label{SubdiffuEst}\delta\norm{W}_{L^2(\Omega)}^2\leq\phi_\tau(W)+c|\Omega|+\norm{g(\tau)}_{L^2(\Omega)}\norm{W}_{L^2(\Omega)}.\end{equation} Consequently, by classical results of subdifferentials (see for instance, \cite{BrezisMMPDE} or \cite{kenmochi1981solvability}), the subdifferential $\partial\phi_t$ is a maximal monotone operator of $H^{-s}(\Omega)$.

In fact, the subdifferential $\partial\phi_t$ is characterized as follows: \begin{equation}\label{SubdiffChar}V\in\partial\phi_t(U)\text{ in }H^{-s}(\Omega) \quad\text{ if and only if }\quad U\in L^2(\Omega)\text{ and }\mathcal{L}^{-1}(V)+g=\gamma(U)\text{ a.e. in }\Omega,\end{equation} and we recall from \eqref{ConvexFunctional} that \[\gamma(U)-g=\mathcal{L}^{-1}(V)=v\in H^s_0(\Omega),\] representing the Dirichlet condition in weak form in the trace sense for $s>\frac{1}{2}$ and more generally $\gamma(U)=g$ in $\Omega^c$. Indeed, the characterisation of the subdifferential in terms of the convex conjugate functions involving $(U,V)$ for $U,V\in H^{-s}(\Omega)$ reads as:  \begin{equation}\label{SubdiffConjugate}V\in\partial\phi_t(U)\iff\phi_t(U)+\phi_t^*(V)=(U,V)\end{equation} where $\phi_t^*(V)=\sup_W\{(W,V)-\phi_t(W)\}$. Then for a given $V\in H^{-s}(\Omega)$,  \begin{align*}\phi_t^*(V)&=\sup_{W\in L^2(\Omega)}\{\langle W,\mathcal{L}^{-1}V\rangle-\phi_t(W)\}\\&=\sup_{W\in L^2(\Omega)}\left\{\langle W,\mathcal{L}^{-1}V\rangle-\int_\Omega j(W)-gW\,dx\right\}\\&=\sup_{W\in L^2(\Omega)}\left\{\int_\Omega W(\mathcal{L}^{-1}V+g)-\int_\Omega j(W)\,dx\right\}.\end{align*} 
Set $J(W)=\int_\Omega j(W)$. Recognising the evaluation at the point $\mathcal{L}^{-1}V+g$ with the convex conjugate on $L^2(\Omega)$ of $j(W)$, by well-known results (see for example Lemma 1 of \cite{Brezis1972ConvexIntegrals}, or \cite{Rockafellar1}), we can associate the convex conjugate $J^*(U)$ with $\int_\Omega j^*(U)$, so we have \[\phi_t^*(V)=\int_\Omega j^*(\mathcal{L}^{-1}V+g)\,dx,\] where $j^*$ is the convex conjugate of $j$ on $\mathbb{R}$. From \eqref{SubdiffConjugate}, this means that \[\int_\Omega j(U)-gU+j^*(\mathcal{L}^{-1}V+g)=\langle\mathcal{L}^{-1}V,U\rangle,\]  or \begin{equation}\label{ConvexFunctionalConjIntegral}\int_\Omega j(U)+j^*(\mathcal{L}^{-1}V+g)-U(\mathcal{L}^{-1}V+g)=0.\end{equation} Recall (see for example, \cite{AttouchBook1984}) that for dual convex functions $j$ and $j^*$, \[j(a)+j^*(b)\geq ab\] for all numbers $a,b$. 
Therefore, the integrand in \eqref{ConvexFunctionalConjIntegral} must be non-negative, and so it is almost everywhere zero, i.e. \[j(U)+j^*(\mathcal{L}^{-1}V+g)-U(\mathcal{L}^{-1}V+g)=0.\] This means that the points $U$ and $\mathcal{L}^{-1}V+g$ are conjugated, i.e. $\mathcal{L}^{-1}V+g\in\partial j(U)$. By definition of $j$ as the primitive of $\gamma$, we have $\mathcal{L}^{-1}V+g=\partial j(U)=\gamma(U)$.

Now, we are ready to apply Proposition \ref{AttouchDamlamianThm} in the space $H^{-s}(\Omega)$ with the convex functions $\phi_t$. For $W\in D(\phi_\tau)\cap D(\phi_t)=D(\phi_0)$ since the domain $D(\phi_t)$ as given in \eqref{ConvexFunctionalDomain} is independent of $t$, we have, by \eqref{ConvexFunctional}, \[\phi_t(W)-\phi_\tau(W)=-\int_\Omega W(g(t)-g(\tau)),\] so, by the Cauchy-Schwarz inequality, \begin{equation}\label{ConvexFunctionalContinuity}|\phi_t(W)-\phi_\tau(W)|\leq\norm{g(t)-g(\tau)}_{L^2(\Omega)}\norm{W}_{L^2(\Omega)}.\end{equation} Also, from \eqref{SubdiffuEst}, we have that \begin{equation}\label{LinftyEst}\norm{W}_{L^2(\Omega)}\leq C_5(1+\phi_\tau(W)),\end{equation} where $C_5$ depends only on $\delta$, $|\Omega|$ and $\norm{g}_{BV(0,T;L^2(\Omega))}$. Therefore, with the given regularity of $g$ inherited from $\tilde{g}$, \eqref{AttouchDamlamianThmHypo} is satisfied, hence we can apply Proposition \ref{AttouchDamlamianThm} to solve the Cauchy problem \begin{equation}\label{CauchyProb}\frac{d\eta}{dt}+\partial\phi_t(\eta(t))\ni f(t)\quad\text{ for almost all }t\in[0,T],\quad \eta(0)=\eta_0\text{ in }H^{-s}(\Omega)\end{equation} with $\eta_0\in D(\phi_0)$, i.e. $\eta_0\in L^2(\Omega)$ and $j(\eta_0)\in L^1(\Omega)$, obtaining a unique \[\eta\in H^1(0,T;H^{-s}(\Omega)).\] Moreover, the estimates in Proposition \ref{AttouchDamlamianThm} and \eqref{LinftyEst} give \[\eta\in L^\infty(0,T;L^2(\Omega)).\]

Also, setting $\vartheta=\gamma(\eta)$ gives \[\mathcal{L}(\vartheta-g)\in L^2(0,T;H^{-s}(\Omega)),\] so that \[\vartheta-g\in L^2(0,T;H^s_0(\Omega)),\] and by \eqref{SubdiffChar}, \[\partial\phi_t(\eta(t))=\mathcal{L}(\gamma(\eta)-g).\] Therefore, multiplying \eqref{CauchyProb} by a test function $\xi\in L^2(0,T;H^s_0(\Omega))$, since $\eta\in H^1(0,T;H^{-s}(\Omega))$, we have \[\left\langle\frac{d \eta}{d t},\xi\right\rangle+\langle \mathcal{L}(\gamma(\eta)-g),\xi\rangle=\langle f,\xi\rangle,\] which is \eqref{FracStefanEta}.

Finally, for $\xi\in \Xi^s_T$, we can integrate in time by parts and obtain \eqref{FracStefan}.
\end{proof}

\begin{remark}
To apply Proposition \ref{AttouchDamlamianThm}, we see from \eqref{ConvexFunctionalContinuity} that it is sufficient to require $g\in BV(0,T;L^2(\Omega))$, as in \cite{Damlamian1977}. However, we require additionally that $g\in L^2(0,T;H^s(\mathbb{R}^d))$ so that \eqref{FracStefanEta}--\eqref{FracStefan} is well-defined.
\end{remark}

\begin{remark} We observe that the general result of the above proposition and theorem applies to general maximal monotone operators of subdifferential type with different functions $\gamma$, and so, besides two-phase Stefan-type problems, it applies also to other models including the porous medium equation. 
In fact, different assumptions on $\gamma$ can be used (see page 12 of \cite{damlamian1976thesis} for more details), generalising the case of the assumption \eqref{gammacond}.
\end{remark}

\begin{remark}\label{AttouchDamlamianThmContinuity}
Considering the above proposition in the case where the Hilbert space $H$ is $H^{-s}(\Omega)$, the solution to the Cauchy problem \eqref{AttouchDalamianThmCauchyEq} in $H^{-s}(\Omega)$ with the convex function $\phi_t$, with domain $L^2(\Omega)$, is obtained by considering the approximated problem with the convex function given by its Yosida approximation $\phi_{t,\lambda}(V)=\frac{1}{\lambda}(Id+(Id+\lambda \partial\phi_t)^{-1})V$. Since the estimate \eqref{AttouchDamlamianThmHypo} carries over to $\phi_{t,\lambda}$, we can apply the Gronwall's inequality to obtain the estimates \eqref{AttouchDalamianThmEst1} and \eqref{AttouchDalamianThmEst3} for the solutions to the approximated problem as in the Part 3 of the proof of Theorem 1 of \cite{attouch1975problemes}. Next, we make use of the absolute continuity of the map $t\mapsto \phi_{t,\lambda}(V)$ to apply to the \eqref{AttouchDalamianThmCauchyEq} to obtain the estimate \eqref{AttouchDalamianThmEst2} from the time derivative. Passing to the limit for the approximated problems give the corresponding constants $C_1$, $C_2$ and $C_3$ for the problem \eqref{AttouchDalamianThmCauchyEq} in $H^{-s}(\Omega)$.

Therefore, for $\sigma\leq s\leq1$, recalling that we have the continuity of the inclusions $H^{-\sigma}(\Omega)\subset H^{-s}(\Omega)\subset H^{-1}(\Omega)$ as a consequence of Lemma \ref{sContDepStefan} below, we can bound the $H^{-s}(\Omega)$ norms with the $H^{-\sigma}(\Omega)$ norms, thereby obtaining the solution to \eqref{AttouchDalamianThmCauchyEq} for all $s$, $\sigma\leq s\leq1$ with the corresponding estimates \eqref{AttouchDalamianThmEst1}--\eqref{AttouchDalamianThmEst3} for the constants $C_1$, $C_2$ and $C_3$ depending only on $\sigma$ and independent of $s$.

Since the constant $C_4$ is obtained from \eqref{AttouchDalamianThmEst3} and \eqref{LinftyEst}, similarly, we can once again consider the problem \eqref{FracStefanEta} in $H^{-s}(\Omega)$ for each $s$, $\sigma\leq s\leq1$, and such that the constant $C_4$ in \eqref{Estimate3} may be chosen depending only on $\sigma$ and not on $s$.
\end{remark}

Furthermore, we have the following continuous dependence result (see also Lemma 3.2 of \cite{DamlamianKenmochi1986Asymptotic}).

\begin{proposition}\label{ContDep}
\sloppy Let $\eta$ and $\hat{\eta}$ denote two generalised enthalpy solutions of the fractional Stefan-type problem \eqref{FracStefanEta} corresponding to $(f,g,\eta_0)$ and $(\hat{f},g,\hat{\eta}_0)$ respectively, where $f,\hat{f},g$, and $\eta_0, \hat{\eta}_0$ are as in the assumptions of Theorem \ref{VarStefanThm1}. Then, for any $0\leq t\leq T$: \begin{equation}\label{ContDepResultEnthalpy}\norm{\eta(t)-\hat{\eta}(t)}_{H^{-s}(\Omega)}\leq\norm{\eta_0-\hat{\eta}_0}_{H^{-s}(\Omega)}+\int_0^t\norm{ f(\tau)-\hat{f}(\tau)}_{H^{-s}(\Omega)}\,d\tau\end{equation} and furthermore \begin{equation}\label{ContDepResultEnthalpyTemp}\norm{\vartheta-\hat{\vartheta}}_{L^2(0,T;L^2(\Omega))}\leq\sqrt{C_\gamma}\norm{\eta_0-\hat{\eta}_0}_{H^{-s}(\Omega)}+\sqrt{\frac{3C_\gamma}{2}}\norm{ f-\hat{f}}_{L^1(0,T;H^{-s}(\Omega))}.\end{equation}
\end{proposition}

\begin{proof}
Writing $\vartheta=\gamma(\eta)$ and $\hat{\vartheta}=\gamma(\hat{\eta})$, we have in $H^{-s}(\Omega)$, \begin{equation}\label{ContDepEqProb1}\frac{d\eta}{d t}(\tau)=-\mathcal{L}_A^s(\vartheta(\tau)-g(\tau))+f(\tau)\end{equation} and \begin{equation}\label{ContDepEqProb2}\frac{d\hat{\eta}}{dt}(\tau)=-\mathcal{L}_A^s(\hat{\vartheta}(\tau)-g(\tau))+\hat{f}(\tau)\end{equation} for a.e. $\tau\in[0,T]$. Taking the difference of these two equations and multiplying by $\eta-\hat{\eta}$, we have \begin{align*}&\,\frac{d}{d \tau}\norm{\eta(\tau)-\hat{\eta}(\tau)}_{H^{-s}(\Omega)}^2=2\left( \eta'(\tau)-\hat{\eta}'(\tau),\eta(\tau)-\hat{\eta}(\tau)\right)\\=&\,-2\left( \mathcal{L}_A^s(\vartheta(\tau)-\hat{\vartheta}(\tau)-g(\tau)+g(\tau)),\eta(\tau)-\hat{\eta}(\tau)\right)+2\left( f(\tau)-\hat{f}(\tau),\eta(\tau)-\hat{\eta}(\tau)\right)\\=&\,-2\left( \mathcal{L}_A^s(\vartheta(\tau)-\hat{\vartheta}(\tau)),\eta(\tau)-\hat{\eta}(\tau)\right)+2\left( f(\tau)-\hat{f}(\tau),\eta(\tau)-\hat{\eta}(\tau)\right)\end{align*} for a.e. $\tau\in[0,T]$. Recalling by Theorem \ref{VarStefanThm1} that $\vartheta(\tau)-\hat{\vartheta}(\tau)\in H^s_0(\Omega)\subset L^2(\Omega)$ and $\eta(\tau)-\hat{\eta}(\tau)\in L^2(\Omega)$ for a.e. $\tau$, observe that the Lipschitz property of $\gamma$ give \begin{align*}\left(\mathcal{L}_A^s(\vartheta(\tau)-\hat{\vartheta}(\tau)),\eta(\tau)-\hat{\eta}(\tau)\right)&=\int_\Omega\left(\vartheta(\tau)-\hat{\vartheta}(\tau)\right)\left(\eta(\tau)-\hat{\eta}(\tau))\right)\geq\frac{1}{C_\gamma}\norm{\vartheta(\tau)-\hat{\vartheta}(\tau)}_{L^2(\Omega)}^2\end{align*} by \eqref{Fdef} and by identifying the duality $\langle\cdot,\cdot\rangle$ with the $L^2(\Omega)$-inner product in the framework of the Gelfand triple $H^s_0(\Omega)\hookrightarrow L^2(\Omega)\hookrightarrow H^{-s}(\Omega)$. Therefore, we deduce that \begin{equation}\label{ContDepEst0}\frac{d}{d \tau}\norm{\eta(\tau)-\hat{\eta}(\tau)}_{H^{-s}(\Omega)}^2+\frac{2}{C_\gamma}\norm{\vartheta(\tau)-\hat{\vartheta}(\tau)}_{L^2(\Omega)}^2\leq2\left( f(\tau)-\hat{f}(\tau),\eta(\tau)-\hat{\eta}(\tau)\right)\end{equation}  for a.e. $\tau\in[0,T]$. Integrating both sides of \eqref{ContDepEst0} over $[0,t]\subset[0,T]$ for any $T>0$ gives \begin{align}\label{ContDepEst}\begin{split}&\,\norm{\eta(t)-\hat{\eta}(t)}_{H^{-s}(\Omega)}^2+\frac{2}{C_\gamma}\int_0^t\norm{\vartheta(\tau)-\hat{\vartheta}(\tau)}_{L^2(\Omega)}^2\,d\tau\\\leq&\,\norm{\eta_0-\hat{\eta}_0}_{H^{-s}(\Omega)}^2+2\int_0^t\left( f(\tau)-\hat{f}(\tau),\eta(\tau)-\hat{\eta}(\tau)\right)\,d\tau\\\leq&\,\norm{\eta_0-\hat{\eta}_0}_{H^{-s}(\Omega)}^2+2\int_0^t\norm{ f(\tau)-\hat{f}(\tau)}_{H^{-s}(\Omega)}\norm{\eta(\tau)-\hat{\eta}(\tau)}_{H^{-s}(\Omega)}\,d\tau\end{split}\end{align} by the Cauchy-Schwarz inequality. Finally, recalling \eqref{Estimate1}, we apply these estimates and a Gronwall-type inequality (see Lemma \ref{GronwallTypeIneq} below) to obtain the result \eqref{ContDepResultEnthalpy}. 

Furthermore, applying the Cauchy-Schwarz inequality again, we obtain, applying \eqref{ContDepResultEnthalpy} to \eqref{ContDepEst}, \begin{align*}\frac{2}{C_\gamma}\int_0^T\norm{\vartheta(t)-\hat{\vartheta}(t)}_{L^2(\Omega)}^2\,dt&\leq\norm{\eta_0-\hat{\eta}_0}_{H^{-s}(\Omega)}^2+2\norm{\eta_0-\hat{\eta}_0}_{H^{-s}(\Omega)}\int_0^T\norm{ f(t)-\hat{f}(t)}_{H^{-s}(\Omega)}\,dt\\&\quad+2\int_0^T\norm{ f(t)-\hat{f}(t)}_{H^{-s}(\Omega)}\left(\int_0^t\norm{ f(\tau)-\hat{f}(\tau)}_{H^{-s}(\Omega)}\,d\tau\right)\,dt\\&\leq\norm{\eta_0-\hat{\eta}_0}_{H^{-s}(\Omega)}^2+2\norm{\eta_0-\hat{\eta}_0}_{H^{-s}(\Omega)}\norm{ f-\hat{f}}_{L^1(0,T;H^{-s}(\Omega))}\\&\quad+2\int_0^T\norm{ f(t)-\hat{f}(t)}_{H^{-s}(\Omega)}\left(\norm{ f-\hat{f}}_{L^1(0,T;H^{-s}(\Omega))}\right)\,dt\\&\leq2\norm{\eta_0-\hat{\eta}_0}_{H^{-s}(\Omega)}^2+3\norm{ f-\hat{f}}_{L^1(0,T;H^{-s}(\Omega))}^2\end{align*} which gives \eqref{ContDepResultEnthalpyTemp}.
\end{proof}

\begin{lemma}\label{GronwallTypeIneq}
Let $F\in L^1(0,T)$ and $y\in L^\infty(0,T)$ be non-negative functions and $C>0$ a constant such that \[y^2(t)\leq \int_0^tF(\tau)y(\tau)\,d\tau+C\quad\text{ for }t\in]0,T[.\] Then we have \[y(t)\leq\frac{1}{2}\int_0^tF(\tau)\,d\tau+\sqrt{C}\quad\text{ for }t\in[0,T].\]
\end{lemma}
\begin{proof}
Let $x(t)=\int_0^tF(\tau)y(\tau)\,d\tau+C$. Then $x'=Fy\leq F\sqrt{x}$. Integrating in time of the relation $\frac{d}{dt}(\sqrt{x})=\frac{x'}{2\sqrt{x}}\leq\frac{F}{2}$, we have the result.
\end{proof}

\begin{remark}\label{EnthalpyContDepGeneral}
In general, for $\gamma\not\equiv\hat{\gamma}$, $g\neq\hat{g}$ and an arbitrary time interval $0\leq t_1<t_2\leq T$, with a similar argument we have the fractional version of the continuous dependence property corresponding to Lemma 3.2 of \cite{DamlamianKenmochi1986Asymptotic} for the classical case $s=1$: \begin{multline}\label{ContDepFull}\norm{\eta(t_2)-\hat{\eta}(t_2)}_{H^{-s}(\Omega)}^2+\frac{2}{C_\gamma}\int_{t_1}^{t_2}\norm{\vartheta(\tau)-\gamma(\hat{\eta})(\tau)}_{L^2(\Omega)}^2\,d\tau+2\int_{t_1}^{t_2}\langle\gamma(\hat{\eta})(\tau)-\hat{\vartheta}(\tau),\eta(\tau)-\hat{\eta}(\tau)\rangle\,d\tau\\\leq\norm{\eta(t_1)-\hat{\eta}(t_1)}_{H^{-s}(\Omega)}^2+2\int_{t_1}^{t_2}\left( f(\tau)-\hat{f}(\tau),\eta(\tau)-\hat{\eta}(\tau)\right)\,d\tau+2\int_{t_1}^{t_2}\int_\Omega (g(\tau)-\hat{g}(\tau))(\eta(\tau)-\hat{\eta}(\tau))\,dx\,d\tau.\end{multline}

As a consequence, we immediately see that if $f=\hat{f}$, $g=\hat{g}$ and $\gamma=\hat{\gamma}$, then \[\norm{\eta(t_2)-\hat{\eta}(t_2)}_{H^{-s}(\Omega)}^2+\frac{2}{C_\gamma}\int_{t_1}^{t_2}\norm{\vartheta(\tau)-\hat{\vartheta}(\tau)}_{L^2(\Omega)}^2\,d\tau\leq\norm{\eta(t_1)-\hat{\eta}(t_1)}_{H^{-s}(\Omega)}^2\] for any $0\leq t_1\leq t_2\leq T$. Furthermore, in this case, the map $t\mapsto\norm{\eta(t)-\hat{\eta}(t)}_{H^{-s}(\Omega)}$ is non-increasing in $t\in[0,T]$ for the same given data. 
\end{remark}

Also as a consequence of \eqref{ContDepFull} with $\gamma=\hat{\gamma}$ and the estimates leading to \eqref{EnthalpyReg} of Theorem \ref{VarStefanThm1}, we have the following corollary:
\begin{corollary}The solution of the variational Stefan-type problem \eqref{FracStefanEta} on the interval $[0,T]$ depends continuously on $f$, $g$ and $\eta_0$ in the following sense: if a sequence $f_m\in L^2(0,T;H^{-s}(\Omega))$, $g_m\in BV(0,T;L^2(\Omega))\cap L^2(0,T;H^{s}(\mathbb{R}^d))$ and $\eta_{0,m}\in L^2(\Omega)$, is such that the $g_m$'s and the $\eta_{0,m}$'s are uniformly bounded in those spaces and $f_m\to f$ in $L^2(0,T;H^{-s}(\Omega))$ and $g_m\to g$ in $L^2(0,T;L^2(\Omega))$ and $\eta_{0,m}\to\eta_0$ in $H^{-s}(\Omega)$, then the solution $\eta_m$ converges to $\eta$  in $L^2(0,T;H^{-s}(\Omega))$ and $\vartheta_m=\gamma(\eta_m)$ converges to $\vartheta=\gamma(\eta)$ in $L^2(0,T;L^2(\Omega))$.\end{corollary}

\section{Regularity of the Weak Temperature Solution $\vartheta$}\label{Sect:TempExist}

If we further assume that $g$ has two time derivatives, by the Lipschitz continuity of $\gamma$, we can achieve higher regularity of the weak temperature solution $\vartheta=\gamma(\eta)$ in \eqref{FracStefan}. The proof makes use of the Faedo-Galerkin method, and follows closely Chapter 6 of \cite{damlamian1976thesis}, and we include it here for completeness.

Let $(F_n)_{n\in\mathbb{N}}$ be an increasing set of finite dimensional subspaces of $H^s_0(\Omega)$, such that their union is dense in $H^s_0(\Omega)$, generated by the eigenvectors of the operator $\mathcal{L}^{-1}|_{L^2(\Omega)}$. This is possible since the inverse of $\mathcal{L}$ is compact in $L^2(\Omega)$, by the compactness of the injection $H^s_0(\Omega)\hookrightarrow L^2(\Omega)$. We denote $F_n^*=\mathcal{L}(F_n)\subset H^{-s}(\Omega)$ and set \[\phi_{t,n}=\phi_t+I_{F_n^*} \text{ in }H^{-s}(\Omega),\] where $I_{F_n^*}$ is the indicator function of $F_n^*$, i.e. $I_{F_n^*}=0$ in $F_n^*$, $I_{F_n^*}=+\infty$ elsewhere.

We first recall a result of Attouch (Theorem 1.10 of \cite{AttouchConvergenceConvexFunctionals}), which relates the Mosco convergence of the convex functionals and the convergence of the solutions of the Cauchy problem in the space $H=H^{-s}(\Omega)$. \begin{proposition}\label{AttouchConvergeThm}
Let $H$ be a real Hilbert space with a scalar product and associated norm. Let $\varphi_n\xrightarrow{M}\varphi$ be a set of lower semi-continuous convex functions in $L^2(0,T;H)$ that converges in the Mosco sense in $H$. Denote $\eta_n$ the solutions of the evolution equations \begin{equation}\frac{d\eta_n}{dt}+\partial\varphi_n(\eta_n)\ni f_n,\quad \eta_n(0)=\eta_{0,n}\end{equation} where $f_n\in L^2(0,T;H)$,  $\eta_{0,n}\in\overline{D(\varphi_n)}$. Suppose that $\eta_{0,n}\to \eta_0$ in $H$, $f_n\to f$ in $L^2(0,T;H)$. Assume also that $\frac{d\eta_n}{dt}$ is bounded in $L^2(0,T;H)$. Then there exists a limit $\eta\in H^1(0,T;H)$, such that $\eta_n\rightharpoonup \eta$ weakly in $H^1(0,T;H)$, where $\eta$ is the solution of \begin{equation}\label{LimitCauchyProb}\frac{d\eta}{dt}+\partial\varphi(\eta)\ni f,\quad \eta(0)=\eta_0.\end{equation}
\end{proposition}

\sloppy With this proposition, our approach would be to determine the subdifferental of $\phi_{t,n}$ and show that they converge to $\phi_t$ in the sense of Mosco. We recall that $\varphi_n\xrightarrow{M}\varphi$ if for every $x\in D(\varphi)$, there exists an approximating sequence of elements $x_n\in D(\varphi_n)$, converging strongly to $x$, such that $\limsup_{n\to \infty }\varphi_n(x_n)\leq \varphi(x)$, and for any subsequence $\varphi_{n_k}$ of $\varphi_n$ such that $x_k\rightharpoonup x$ in $H$, we have $\liminf_{k\to \infty }\varphi_{n_k}(x_k)\geq \varphi(x)$.
Then applying Proposition \ref{AttouchConvergeThm} to our Faedo-Galerkin approximation, and with the additional estimates we obtain from Proposition \ref{AttouchDamlamianThm}, we can pass to the limit to get the additional regularity to the solution for the limit problem. 

For simplicity, we drop the parameter $t$ and consider $t$ to be fixed in $]0,T[$, and we denote $\phi_{t,n}$ as $\phi_n$ and $\phi_t=\phi$. Denote $i$ to be the compact injection of $H^s_0(\Omega)$ into $L^2(\Omega)$ and take $E_n=i(F_n)$ by considering $F_n$ as a subspace of $H^s_0(\Omega)$. It is clear that $i^{-1}$ is an isomorphism between $E_n$ and $F_n$, with norm depending on $n$. 

\begin{proposition}\label{MoscoConvgConvexFuncTemp}
$\phi_n\xrightarrow{M}\phi$ in $H^{-s}(\Omega)$.
\end{proposition}

\begin{proof}
Denote $i^*$ to be the injection map from $L^2(\Omega)$ to $H^{-s}(\Omega)$. Then $i^*(E_n)=i^*\circ i(F_n)=F_n^*$. Indeed, for an eigenvector $u$ of $\mathcal{L}|_{L^2(\Omega)}$ corresponding to an eigenvalue $\mu$, we have, by definition, $\mathcal{L}u=\mu i^*\circ i(u)$ in $H^{-s}(\Omega)$, hence the result.

For $i^*(U)\in D(\phi)$, we define $i^*(U_n)$, where $U_n=\mathbb{P}_{E_n}U$ is the projection of $U$ into $E_n$ in $L^2(\Omega)$. Since $\overline{\cup E_n}^{L^2}=L^2(\Omega)$ by construction, so $U_n\to U$ in $L^2(\Omega)$, and therefore $i^*(U_n)=\mathbb{P}_{F_n^*}i^*(U)\to i^*(U)$ in $H^{-s}(\Omega)$.

In addition, since $\gamma$ satisfies the growth condition \eqref{gammacond} at $\pm\infty$, its primitive $j$ is quadratic at $\pm\infty$ (so that $j(r)/|r|^2$ and its inverse remain bounded as $r\to\pm\infty$). Therefore, by the dominated convergence theorem, the map $U\mapsto\int_\Omega j(U)$ is continuous in $L^2(\Omega)$, and since $i^*(U_n)\in F_n^*$, so $\phi_n(i^*(U_n))=\phi(i^*(U_n))\to\phi(i^*(U))$.

On the other hand, the sequence $\phi_n$ is decreasing (since $F_n$ is increasing), so we conclude the Mosco convergence of $\phi_n$ to $\phi$ given that $\phi$ is known to be lower semi-continuous.
\end{proof}

Next, we want to obtain a solution of the approximate Cauchy problem for $\eta_n$, making use of Proposition \ref{AttouchDamlamianThm} as in the proof of Theorem \ref{VarStefanThm1}.

\begin{proposition}\label{vnChar} Setting $V=\mathcal{L}v$,
\[V\in\partial\phi_n(U)\text{ in }H^{-s}(\Omega)\text{ if and only if }U\in D(\phi)\cap F_n^*,  \gamma(U)-g\in L^2(\Omega)\text{ and }i(v)+g-\gamma(U)\perp E_n\text{ in }L^2(\Omega).\] 
\end{proposition}

\begin{proof}
Denote the inf-convolution of two convex functions by the composition operator $\nabla$. Then by definition, we know that the convex conjugate $\phi_n^*=(\phi^*\nabla I^*_{F_n^*})^{**}$, where the double asterisk $^{**}$ stands for the regularized l.s.c. function of $\psi_n=\phi^*\nabla I^*_{F_n^*}$.

Since $F_n^*$ is a subspace of $H^{-s}(\Omega)$, we have $I^*_{F_n^*}=I_{(F_n^*)^\perp}$, where the orthogonality is inherited from the duality between $H^s_0(\Omega)$ and $H^{-s}(\Omega)$. Since $\mathcal{L}(F_n)=F_n^*$, $(F_n^*)^\perp$ is also the orthogonal of $F_n$ in $H^s_0(\Omega)$. We therefore have \[\psi_n(w)=\phi^*\nabla I^*_{F_n^*}(w)=\phi^*\nabla I_{(F_n^*)^\perp}(w)=\inf_{\mathbb{P}_{F_n}(z-w)=0}\int_\Omega j^*(g+z).\] Since $\gamma$ is globally Lipschitz, $\beta$ satisfies the growth assumption \eqref{gammacond} at infinity, so the function $j^*$ is quadratic at infinity and therefore $z\mapsto\int_\Omega j^*(z)$ is continuous in $L^2(\Omega)$. Furthermore, the function $z\mapsto\int_\Omega j^*(z)$ is coercive in $L^2(\Omega)$. 

Henceforth, we deduce that there exists $z=z(v)$ in $L^2(\Omega)$, not necessarily unique, such that $\psi_n(v)=\int_\Omega j^*(g+z(v))$ with $z(v)-i(v)$ in $L^2(\Omega)$, such that $z(v)-i(v)\perp E_n$ in $L^2(\Omega)$. Indeed, $z-v\perp F_n$ in $H^s_0(\Omega)$ so $\langle \mathcal{L}\xi,z-v\rangle=0$ for all $\xi$ in the basis of $F_n$. Hence, taking a vector $\xi$ in that basis, we have $\mathcal{L}\xi=i^*(\mathcal{L}\xi)=\mu i^*\circ i(\xi)$, so $0=\int_\Omega i(\xi)i(z-v)$ which means that $i(z)-i(v)$ is orthogonal to $E_n$ in $L^2(\Omega)$. Since $z(v)$ is the weak limit in $L^2(\Omega)$, considering a minimising sequence of such $i(z)$, we have the result.

Futhermore, using the coercivity of the integral of $j^*$ in $L^2(\Omega)$ again, we see that $\psi_n$ is lower semi-continuous in $H^s_0(\Omega)$, so $\psi_n=\phi_n^*$.

Therefore, $V\in\partial\phi_n(U)$ if and only if $i^*(U)\in F_n^*$, and there exists $z(v)\in L^2(\Omega)$ with $z(v)-i(v)\in E_n^\perp$ and, as in \eqref{ConvexFunctionalConjIntegral}, \[\int_\Omega j(U)+j^*(g+z)=\langle U,g+z\rangle.\] But since $U\in D(\phi)\cap F_n^*\subset L^2(\Omega)$, we can rewrite this as \[\int_\Omega j(U)+j^*(g+z)-U(g+z)=0,\] so, as in the proof of Theorem \ref{VarStefanThm1}, we have that the points $U$ and $g+z$ are conjugated by $j$, thus $z(v)+g=\partial j(U)=\gamma(U)$ a.e. in $\Omega$. The reverse is also clearly true.
\end{proof}

Now, setting $f_n=\mathbb{P}_{E_n}f$ for $f\in L^2(0,T;L^2(\Omega))$ and for $\eta_0\in L^2(\Omega)$, we apply the Proposition \ref{AttouchDamlamianThm} for $(\phi_{t,n})_{t\in[0,T]}$ to solve \begin{equation}\label{ApproxCauchyProb}\frac{d\eta_n}{dt}+\partial\phi_{t,n}(\eta_n)\ni i^*(f_n),\quad \eta_n(0)=\eta_{0,n},\end{equation} where $\eta_{0,n}$ is constructed as in the proof of Proposition \ref{MoscoConvgConvexFuncTemp} such that $\eta_{0,n}\in D(\phi_{0,n})$ with $\eta_{0,n}\to \eta_0\in D(\phi_0)$  strongly in $H^{-s}(\Omega)$ and $\phi_{0,n}(\eta_{0,n})\to\phi_0(\eta_0)$. Then by \eqref{AttouchDalamianThmEst2}, $\frac{d\eta_n}{dt}$ is bounded in $L^2(0,T;H^{-s}(\Omega))$. Moreover, as in Proposition \ref{MoscoConvgConvexFuncTemp}, for all $U\in L^2(0,T;H^{-s}(\Omega))$, we have \[\varphi_n(U):=\int_0^T\phi_{t,n}(U(t))\,dt\xrightarrow{M}\varphi(U):=\int_0^T\phi_t(U(t))\,dt\] in the sense of Mosco.

Therefore, applying Proposition \ref{AttouchConvergeThm}, we conclude that $\eta_n$ converges weakly in $H^1(0,T;H^{-s}(\Omega))$ to the solution $\eta$ of \[\frac{d\eta}{dt}+\partial\phi_t(\eta)\ni i^*(f),\quad \eta(0)=\eta_0.\]

Having obtained the approximation $\eta_n\rightharpoonup \eta$ for the enthalpy $\eta$, we want to pass to the limit in the temperatures $\vartheta_n=\gamma(\eta_n)\to\vartheta=\gamma(\eta)$. To do so, we require some estimates on the derivative of the temperatures.

\begin{proposition}\label{TimeDefEstProp}
Suppose $f\in L^2(0,T;L^2(\Omega))$ and $\tilde{g}\in W^{2,1}(0,T;L^2(\mathbb{R}^d))\cap L^\infty(0,T;H^s(\mathbb{R}^d))$. Assume $\eta_0\in L^2(\Omega)$ and, setting $\vartheta(0)=\gamma(\eta_0)$, assume $\vartheta(0)-g(0)\in H^s_0(\Omega)$. Denote by $\eta_n\in H^1(0,T;F_n^*)$, and $\tilde{\eta}_n\in H^1(0,T;E_n)$ such that $\eta_n=i^*(\tilde{\eta}_n)$, the generalised solution associated to the approximate Cauchy problem \eqref{ApproxCauchyProb}, corresponding to the Faedo-Galerkin method as described above. Then, the integral \begin{equation}\label{TimeDefEstPropBounds}\int_0^T\int_\Omega\left|\frac{\partial \gamma(\tilde{\eta}_n)}{\partial t}\right|^2\leq C_6,\quad \norm{\mathbb{P}_{F_n}(\gamma(\tilde{\eta}_n)-g)}_{L^\infty(0,T;H^s_0(\Omega))}\leq C_7\end{equation} is uniformly bounded in $n$, with the bounds $C_6$, $C_7$ dependent on the Lipschitz constant $C_\gamma$ and the given data $f,g,\eta_0$.
\end{proposition}

\begin{proof}
Since $\eta_n\in H^1(0,T;F_n^*)$, there exists $\tilde{\eta}_n\in H^1(0,T;E_n)$ such that $\eta_n=i^*(\tilde{\eta}_n)$, $v_n=\gamma(\tilde{\eta}_n)-g$, and, by Proposition \ref{vnChar} applied to $\tilde{\eta}_n$, satisfies \begin{equation}\label{A1CauchyProbSeq}\frac{\partial \eta_n}{\partial t}+\mathcal{L}v_n=i^*(f_n), v_n\in L^2(0,T;H^s_0(\Omega))\text{  with }i(v_n)+g-\gamma(\tilde{\eta}_n)\perp E_n\text{ in }L^2(\Omega).\end{equation}

Since $\gamma$ is Lipschitz, we have $\gamma(\tilde{\eta}_n)\in H^1(0,T;L^2(\Omega))$ and $\gamma(\tilde{\eta}_n)-g\in H^1(0,T;L^2(\Omega))$. Let $h_n=\mathbb{P}_{F_n}v_n$ and $\tilde{h}_n=\mathbb{P}_{E_n}(\gamma(\tilde{\eta}_n)-g)$. Then \begin{equation}\label{A1ProjSpace}h_n\in H^1(0,T;F_n)\text{ and }\tilde{h}_n\in H^1(0,T;E_n).\end{equation} Indeed, we have $\gamma(\tilde{\eta}_n)-g\in H^1(0,T;L^2(\Omega))$, so $\tilde{h}_n=\mathbb{P}_{E_n}(\gamma(\tilde{\eta}_n)-g)\in H^1(0,T;E_n)$. Since $\tilde{h}_n=\mathbb{P}_{E_n}i(v_n)$, so by the choice of $F_n$, we have $\mathbb{P}_{E_n}\circ i=i\circ\mathbb{P}_{F_n}$, and we deduce that $i(h_n)=\tilde{h}_n$. Therefore, since $i$ gives an isomorphism between $F_n$ and $E_n$, we obtain the properties in \eqref{A1ProjSpace}.

Making use of these properties, we can therefore multiply \eqref{A1CauchyProbSeq} by $\frac{\partial h_n}{\partial t}\in L^2(0,T;F_n)$ to obtain \begin{equation}\label{A1CauchyProbSeqVar1}\int_\Omega\frac{\partial \tilde{\eta}_n}{\partial t}\left[i\left(\frac{\partial h_n}{\partial t}\right)\right]+\left[v_n,\frac{\partial h_n}{\partial t}\right]_A=\int_\Omega f_n\left[
i\left(\frac{\partial h_n}{\partial t}\right)\right]\end{equation} by \eqref{RieszRep}. Now, \[\left[\frac{\partial h_n}{\partial t},v_n\right]_A=\left[\frac{\partial h_n}{\partial t},\mathbb{P}_{F_n}v_n\right]_A=\left[\frac{\partial h_n}{\partial t},h_n\right]_A=\int_{\mathbb{R}^d}AD^s\frac{\partial h_n}{\partial t}\cdot D^sh_n=\frac{1}{2}\frac{\partial }{\partial t}\int_{\mathbb{R}^d}AD^sh_n\cdot D^sh_n,\] and from $f_n=\mathbb{P}_{E_n}f$, we obtain \begin{equation}\label{A1CauchyProbSeqVar2}\int_\Omega\frac{\partial \tilde{\eta}_n}{\partial t}\frac{\partial \tilde{h}_n}{\partial t}+\frac{1}{2}\frac{\partial }{\partial t}\int_{\mathbb{R}^d}AD^sh_n\cdot D^sh_n=\int_\Omega f_n\frac{\partial \tilde{h}_n}{\partial t}.\end{equation} Now, recalling the definition of $\tilde{h}_n$, we observe that \[\int_\Omega\left(\frac{\partial \tilde{\eta}_n}{\partial t}-f_n\right)\frac{\partial \tilde{h}_n}{\partial t}=\int_\Omega\left(\frac{\partial \tilde{\eta}_n}{\partial t}-f_n\right)\frac{\partial }{\partial t}\mathbb{P}_{E_n}(\gamma(\tilde{\eta}_n)-g)=\int_\Omega\left(\frac{\partial \tilde{\eta}_n}{\partial t}-f_n\right)\mathbb{P}_{E_n}\frac{\partial }{\partial t}(\gamma(\tilde{\eta}_n)-g),\] so since $\frac{\partial \tilde{\eta}_n}{\partial t}-f_n\perp L^2(\Omega)\backslash E_n$, we have \begin{equation}\label{A1CauchyProbSeqVar3}\int_\Omega\left(\frac{\partial \tilde{\eta}_n}{\partial t}-f_n\right)\frac{\partial }{\partial t}(\gamma(\tilde{\eta}_n)-g)+\frac{1}{2}\frac{\partial }{\partial t}\int_{\mathbb{R}^d}AD^sh_n\cdot D^sh_n=0.\end{equation} Integrating this over $[0,t]$ for $t\leq T$, we obtain, by the coercivity of $A$ in \eqref{Acoer} and integrating by parts in time, \begin{align}\label{A1CauchyProbSeqVar4}\begin{split}&\int_0^t\int_\Omega\frac{\partial \tilde{\eta}_n}{\partial t}\frac{\partial \gamma(\tilde{\eta}_n)}{\partial t}+\frac{1}{2}a_*\norm{h_n(t)}_{H^s_0(\Omega)}^2\\\leq&\,\frac{1}{2}a^*\norm{h_n(0)}_{H^s_0(\Omega)}^2+\int_0^t\int_\Omega f_n\frac{\partial \gamma(\tilde{\eta}_n)}{\partial t}+\int_0^t\int_\Omega \frac{\partial \tilde{\eta}_n}{\partial t}\frac{\partial g}{\partial t}-\int_0^t\int_\Omega f_n\frac{\partial g}{\partial t}\\=&\,\frac{1}{2}a^*\norm{h_n(0)}_{H^s_0(\Omega)}^2+\int_0^t\int_\Omega f_n\frac{\partial \gamma(\tilde{\eta}_n)}{\partial t}-\int_0^t\int_\Omega \tilde{\eta}_n\frac{\partial^2 g}{\partial t^2}\\&+\int_\Omega \tilde{\eta}_n(t)\frac{\partial g}{\partial t}(t)-\int_\Omega \tilde{\eta}_n(0)\frac{\partial g}{\partial t}(0)-\int_0^t\int_\Omega f_n\frac{\partial g}{\partial t}.\end{split}\end{align}

Now, we know by \eqref{AttouchDalamianThmEst3} that $\phi_{t,n}(\eta_n(t))$ is bounded independent of $n$ and $t$, so $\norm{\tilde{\eta}_n}_{L^\infty(0,T;L^2(\Omega))}$ is bounded independent of $n$ (see also \eqref{Estimate1}). Then, by the Cea-type lemma (see, for instance, Proposition 2.5 of \cite{arendt2019galerkin}) given by  \[\norm{\mathbb{P}_{F_n}w}_{H^s_0(\Omega)}^2\leq\frac{a^*}{a_*} \norm{w}_{H^s_0(\Omega)}^2\quad\forall w\in H^s_0(\Omega),\] we have, by the compatibility of the initial condition giving $h_n(0)=\mathbb{P}_{F_n}(\gamma(\tilde{\eta}_n(0))-g(0))=\mathbb{P}_{F_n}(\vartheta(0)-g(0))$, \begin{multline}\label{A1CauchyProbSeqVar5}\int_0^t\int_\Omega\frac{\partial \tilde{\eta}_n}{\partial t}\frac{\partial \gamma(\tilde{\eta}_n)}{\partial t}+\frac{1}{2}a_*\norm{h_n(t)}_{H^s_0(\Omega)}^2\\\leq\frac{1}{2}\frac{{a^*}^2}{a_*}\norm{\vartheta(0)-g(0)}_{H^s_0(\Omega)}^2+\int_0^t\int_\Omega f_n\frac{\partial \gamma(\tilde{\eta}_n)}{\partial t}-\int_0^t\int_\Omega \tilde{\eta}_n\frac{\partial^2 g}{\partial t^2}\\+\int_\Omega \tilde{\eta}_n(t)\frac{\partial g}{\partial t}(t)-\int_\Omega \tilde{\eta}_n(0)\frac{\partial g}{\partial t}(0)-\int_0^t\int_\Omega f_n\frac{\partial g}{\partial t}.\end{multline}

Now, letting $C_\gamma$ be the Lipschitz constant of $\gamma$, we have \begin{equation}\label{LipschitzTimeDef}\frac{\partial \tilde{\eta}_n}{\partial t}\frac{\partial \gamma(\tilde{\eta}_n)}{\partial t}\geq\frac{1}{C_\gamma}\left|\frac{\partial \gamma(\tilde{\eta}_n)}{\partial t}\right|^2\text{ a.e. } Q_T.\end{equation} Also, observe the boundedness of $\tilde{\eta}_n$ in $L^\infty(0,T;L^2(\Omega))$, since $\eta_n$ is obtained as a solution to the Faedo-Galerkin finite dimensional approximated problem \eqref{ApproxCauchyProb} and therefore also satisfies \eqref{Estimate3}. Therefore, applying the Cauchy-Schwarz inequality to the term $\int_0^t\int_\Omega f_n\frac{\partial \gamma(\tilde{\eta}_n)}{\partial t}$ and making use of the assumption $\vartheta(0)-g(0)\in H^s_0(\Omega)$ gives the first uniform bound $\int_0^T\int_\Omega\left|\frac{\partial \gamma(\tilde{\eta}_n)}{\partial t}\right|^2\leq C_6$.

Using again \eqref{A1CauchyProbSeqVar5}, we can easily take the supremum over all time to obtain the second uniform bound $\norm{\mathbb{P}_{F_n}(\gamma(\tilde{\eta}_n)-g)}_{L^\infty(0,T;H^s_0(\Omega))}=\norm{h_n}_{L^\infty(0,T;H^s_0(\Omega))}\leq C_7$.
\end{proof}

\begin{remark}\label{GalerkinEstRemark}
For fixed $\sigma>0$ and $s$ such that $\sigma\leq s\leq 1$, similarly to Remark \ref{AttouchDamlamianThmContinuity}, we observe that $\tilde{\eta}_n\in L^\infty(0,T;L^2(\Omega))$, and $\tilde{\eta}_n$ can be bounded for each $s$ by a constant depending on $\sigma$ but independent of $s$, by the continuity of the eigenfunctions (in Appendix \ref{Sect:Eigen}), and depending explicity on $T$ and $\gamma$. Similarly, by Appendix \ref{Sect:Eigen}, the $\eta_n$'s are bounded independent of $s\geq\sigma$ in $H^1(0,T;F_n^*)$. This allows us to consider the convergence of the variational problem as $s$ varies.

In addition, when we have a sequence of Lipschitz functions $\gamma_n$, we can also obtain \eqref{LipschitzTimeDef} by considering a Lipschitz constant $C_\gamma$ given by the supremum of all the Lipschitz constants $C_{\gamma_n}$.
\end{remark}

Now, we can finally proceed to show the existence of more regular solutions to the variational problem \eqref{FracStefan}. Indeed, we have the following result:

\begin{theorem}\label{VarStefanThm2HigherReg}
Let $f\in L^2(0,T;L^2(\Omega))$ and $\tilde{g}\in W^{2,1}(0,T;L^2(\mathbb{R}^d))\cap L^\infty(0,T;H^s(\mathbb{R}^d))$, and define $g$ as in \eqref{DirichletBdryCond} with the same regularity (see Appendix \ref{Sect:DirBdryCondAp}). Assume $\eta_0\in L^2(\Omega)$, and setting $\vartheta(0)=\gamma(\eta_0)$ assume $\vartheta(0)-g(0)\in H^s_0(\Omega)$. Then there exists a unique weak temperature solution $\vartheta$ to the variational problem \eqref{FracStefanEta}--\eqref{FracStefan}, such that \begin{equation}\label{TempReg} \vartheta\in L^\infty(0,T;H^s(\mathbb{R}^d))\cap H^1(0,T;L^2(\Omega)).\end{equation}
\end{theorem}

\begin{proof}
\sloppy From Proposition \ref{TimeDefEstProp},  $h_n$ is bounded in $L^\infty(0,T;H^s_0(\Omega))$. Furthermore, if we recall the definition of $\tilde{h}_n$ as the projection onto $E_n$, we have \[\norm{\frac{\partial \tilde{h}_n}{\partial t}}_{L^2(\Omega)}\leq\norm{\frac{\partial }{\partial t}(\gamma(\tilde{\eta}_n)-g)}_{L^2(\Omega)},\] so $\tilde{h}_n$ is bounded in $H^1(0,T;L^2(\Omega))$.

By Proposition \ref{AttouchConvergeThm}, we know that $\eta_n\rightharpoonup \eta$ in $H^1(0,T;H^{-s}(\Omega))$ and $\tilde{\eta}_n\rightharpoonup \eta$ weakly$^*$ in $L^\infty(0,T;L^2(\Omega))$, and \[\mathcal{L}(v_n)=\partial\phi_{t,n}(\eta^n)=i^*(f_n)-\frac{\partial \eta_n }{\partial t}\rightharpoonup i^*(f)-\frac{\partial \eta}{\partial t}=\partial\phi_t(\eta)=\mathcal{L}v\text{ in }L^2(0,T;H^{-s}(\Omega)).\] Therefore, on applying $\mathcal{L}^{-1}$, $v_n$ tends to $v=\gamma(\eta)-g$ weakly in $L^2(0,T;H^s_0(\Omega))$.

Since $v_n=h_n+k_n$ for some $k_n\in F_n^\perp$, we deduce that $k_n\rightharpoonup0$ in $L^2(0,T;H^s_0(\Omega))$ and $h_n\rightharpoonup \gamma(\eta)-g$ also in this space, so $\tilde{h}_n\rightharpoonup i(\gamma(\eta)-g)$ in $L^2(0,T;L^2(\Omega))$. Therefore, by \eqref{TimeDefEstPropBounds}, $\gamma(\eta)-g$ lies in $L^\infty(0,T;H^s_0(\Omega))$ and $i(\gamma(\eta)-g)\in H^1(0,T;L^2(\Omega))$. Finally as $\vartheta=\gamma(\eta)$, we have the desired regularity \eqref{TempReg}.
\end{proof}

\begin{remark}\label{TempSBoundedRemark}
It can be seen that the bounds in \eqref{TimeDefEstPropBounds} can be made to depend only on $\sigma>0$ and independent of $s$ for $\sigma\leq s\leq1$, by the continuity of the eigenfunctions as shown in Appendix \ref{Sect:Eigen}. Then, as in Remark \ref{GalerkinEstRemark}, the bounds $\norm{D^s(\vartheta-g)}_{L^\infty(0,T;L^2(\mathbb{R}^d)^d)}$ and $\norm{\frac{\partial \vartheta}{\partial t}}_{L^2(0,T;L^2(\Omega))}$in \eqref{TempReg} only depend only on $\sigma$ and independent of $s$, allowing us to consider the convergence of the variational problem as $s$ varies.
\end{remark}

\section{Convergence to the Classical Problem as $s\nearrow1$}\label{Sect:sConvg}

Next, as $s\nearrow1$ the $s$-fractional derivatives converge to the classical derivatives, we show that the corresponding solutions to the fractional Stefan-type problem converge in appropriate spaces to the classical one. We first recall the fractional Poincar\'e inequality.

\begin{lemma}[Fractional Poincar\'e inequality, Theorem 2.9 of \cite{BellidoCuetoMoraCorral2021CVPDEgammaconvg}]\label{Poincare}
Let $s\in(0,1)$. Then there exists a constant $C_P=C(d,\Omega)>0$ such that \[\norm{u}_{L^2(\Omega)}\leq \frac{C_P}{s}\norm{D^su}_{L^2(\mathbb{R}^d)^d}\] for all $u\in H^s_0(\Omega)$.
\end{lemma}

To consider the convergence of the problem as $s\nearrow1$, we start with a continuous dependence property of the Riesz derivatives as $s$ varies, which can be easily shown using Fourier transform first for $u(t)\in C_c^\infty(\Omega)$, and extended by density as in Lemma 3.7 of \cite{FracObsRiesz}.

\begin{lemma}\label{sContDepStefan}
For $u\in L^\infty(0,T;H^{s'}_0(\Omega))$, $D^su$ is continuous in $L^\infty(0,T;L^2(\mathbb{R}^d)^d)$ as $s$ varies in $[\sigma,s']$ for $0<\sigma< s'\leq 1$. As a consequence, we have the following estimate: for $\sigma\leq s\leq1$,  \begin{equation}\label{ConvergenceEq1}\norm{D^\sigma u(t)}_{L^2(\mathbb{R}^d)^d}\leq c_\sigma\norm{D^su(t)}_{L^2(\mathbb{R}^d)^d},\end{equation} for any $u(t)\in H^{s}_0(\Omega)$ for a.e. $t\in[0,T]$, where the constant $c_\sigma$ is independent of $s$ and $t$.
\end{lemma}

Consequently, we have a continuous transition from the fractional Stefan-type problem to the classical Stefan-type problem as $s\nearrow1$ in the following sense.

\begin{theorem}\label{ConvergenceThmSStefan}
\sloppy Let $(\eta_s,\vartheta_s)$ be the solution to the fractional Stefan-type problem for $0<\sigma\leq s<1$ for $f_s\in L^2(0,T;L^2(\Omega))$, $\tilde{g}_s\in W^{2,1}(0,T;L^2(\mathbb{R}^d))\cap L^\infty(0,T;H^s(\mathbb{R}^d))$, i.e. $\vartheta_s=\gamma(\eta_s)$ for a.e. $x,t\in Q_T$ and \begin{equation}-\int_{Q_T}\eta_s\frac{\partial \xi}{\partial t}+\int_{\mathbb{R}^d\times[0,T]}AD^s\vartheta_s\cdot D^s\xi=\int_{Q_T} f_s\xi+\int_\Omega\eta_{0,s}\xi(0),\quad\forall \xi\in \Xi^s_T\tag{\ref{FracStefan}}\end{equation} with Dirichlet boundary condition $\vartheta_s=g_s$ on $\Omega^c\times]0,T[$, initial condition $\eta_s(0)=\eta_{0,s}\in L^2(\Omega)$,
and setting $\vartheta_s(0)=\gamma(\eta_{0,s})$ assume $\vartheta_s(0)-g_s(0)\in H^s_0(\Omega)$ is bounded uniformly in $s$ for $0<\sigma\leq s<1$. Suppose that there exists $\eta_0\in L^2(\Omega)$, $f\in L^2(0,T;L^2(\Omega))$ and $\tilde{g}\in W^{2,1}(0,T;L^2(\mathbb{R}^d))\cap L^\infty(0,T;H^1(\mathbb{R}^d))$ such that  \begin{align}\label{sConvgAssumpConvg}\begin{split}\eta_{0,s}\rightharpoonup\eta_0\text { in }&L^2(\Omega),\\\quad f_s\rightharpoonup f\text{ in }&L^2(0,T;L^2(\Omega)),\quad\text{ and }\\ \tilde{g}_s\rightharpoonup \tilde{g}\text{ in }&W^{2,1}(0,T;L^2(\mathbb{R}^d))\text {-weak and in  }L^\infty(0,T;H^\sigma(\mathbb{R}^d))\text {-weak$^*$}.\end{split}\end{align} Then, the sequence $(\eta_s,\vartheta_s)_s$ converges weakly to $(\eta,\vartheta)$ in the sense that \begin{equation}\label{etasconvg}\eta_s\rightharpoonup \eta\text{ in } L^\infty(0,T;L^2(\Omega))\text{-weakly$^*$ and in } H^1(0,T;H^{-1}(\Omega))\text{-weak},\end{equation} and \begin{equation}\label{thetasconvg}\vartheta_s\rightharpoonup\vartheta \text{ in } L^\infty(0,T;H^\sigma(\Omega))\text{-weak$^*$},\text{ in } H^1(0,T;L^2(\Omega))\text{-weak}\text{ and in }C([0,T];L^2(\Omega))\end{equation} as $s\nearrow1$, where $(\eta,\vartheta)$ solves uniquely the Stefan problem for $s=1$ with $\vartheta=\gamma(\eta)$ and initial condition $\eta(0)=\eta_0$ in $\Omega$, and Dirichlet boundary condition $\vartheta=g$ on $\partial\Omega\times]0,T[$, and \begin{equation}\label{ClassicalStefanSLimit}-\int_{Q_T}\eta\frac{\partial \xi}{\partial t}+\int_{Q_T}AD\vartheta\cdot D\xi=\int_{Q_T} f\xi+\int_\Omega\eta_0\xi(0),\quad\forall \xi\in \Xi^1_T.\end{equation}
\end{theorem}

\begin{proof}
Recall that $\eta_s\in L^\infty(0,T;L^2(\Omega))\cap H^1(0,T;H^{-1}(\Omega))$ independent of $\phi_t$, by Remark \ref{AttouchDamlamianThmContinuity}. Moreover, invoking the continuity of the inclusions $H^{-\sigma}(\Omega)\subset H^{-s}(\Omega)\subset H^{-1}(\Omega)$, we have, by \eqref{Estimate3}, \begin{equation}\label{EstEq1}\norm{\eta_s}_{L^\infty(0,T;L^2(\Omega))}\leq C_4\left(\norm{\eta_{0,s}}_{L^2(\Omega)}, \norm{f_s}_{L^2(0,T;H^{-s}(\Omega))},\norm{g_s}_{BV(0,T;L^2(\Omega))}\right)\leq C_4'\end{equation} for a constant $C_4'$ depending on $\sigma$ but independent of $s$ by assumption \eqref{sConvgAssumpConvg}. Then, by Lemma \ref{Poincare}, \[\norm{f_s}_{L^2(0,T;H^{-s}(\Omega))}\leq \frac{C_P}{\sigma}\norm{f_s}_{L^2(0,T;L^2(\Omega))}.\]  Similarly, by \eqref{Estimate2} and \eqref{sConvgAssumpConvg}, \begin{equation}\label{EstEq2}\frac{1}{c_1}\norm{\frac{\partial \eta_s}{\partial t}}_{L^2(0,T;H^{-1}(\Omega))}\leq\norm{\frac{\partial \eta_s}{\partial t}}_{L^2(0,T;H^{-s}(\Omega))}\leq C_2\left(\norm{\eta_{0,s}}_{L^2(\Omega)}, \norm{f_s}_{L^2(0,T;L^2(\Omega))},\norm{g_s}_{BV(0,T;L^2(\Omega))}\right).\end{equation} Therefore, $\eta_s$ is bounded in $L^\infty(0,T;L^2(\Omega))\cap H^1(0,T;H^{-1}(\Omega))$ uniformly with respect to $s$, and, up to a subsequence, $(\eta_s)_s$ is converging in $H^1(0,T;H^{-1}(\Omega))$-weak and in $L^\infty(0,T;L^2(\Omega))$-weak$^*$ to some $\eta$ as in \eqref{etasconvg}.

Furthermore, for $\vartheta_s-g_s\in L^\infty(0,T;H^s_0(\Omega))$, we have  \begin{equation}\label{ConvergenceEq1Stefan}\norm{D^s(\vartheta_s-g_s)}_{L^\infty(0,T;L^2(\mathbb{R}^d)^d)}\leq C\end{equation} by \eqref{sConvgAssumpConvg} and Remarks \ref{GalerkinEstRemark} and \ref{TempSBoundedRemark} for some constant $C$ independent of $s$ depending on $\sigma\leq s$ and on the data. By the Poincar\'e inequality, $\vartheta_s-g_s$ is also bounded, so \[\vartheta_s-g_s\xrightharpoonup[s\nearrow1]{}\vartheta-g \text{ in }L^\infty(0,T;L^2(\mathbb{R}^d))\text{-weak$^*$}\quad\text{ and }\quad D^s(\vartheta_s-g_s)\xrightharpoonup[s\nearrow1]{}\zeta \text{ in }L^\infty(0,T;L^2(\mathbb{R}^d)^d)\text{-weak$^*$}\] for some $\vartheta,\zeta$. 

\sloppy Now, by the convergence Lemma \ref{sContDepStefan}, for all $\Phi\in L^2(0,T;C_c^\infty(\Omega)^d)$, denoting by $\tilde{\Phi}$ the zero extension of $\Phi$ outside $\Omega$, \[D^s\cdot\Phi\xrightarrow[s\nearrow1]{} \widetilde{D\cdot\Phi}=D\cdot\tilde{\Phi}\quad\text{ in }L^2(0,T;L^2(\mathbb{R}^d)^d),\] therefore, \[\int_0^T\int_{\mathbb{R}^d} D^s(\vartheta_s-g_s)\cdot\tilde{\Phi}=-\int_0^T\int_{\mathbb{R}^d} (\vartheta_s-g_s)(D^s\cdot\Phi)\xrightarrow[s\nearrow1]{}-\int_0^T\int_{\mathbb{R}^d} (\vartheta-g)\widetilde{(D\cdot\Phi)}.\] 
But by the a priori estimate on $D^s(\vartheta_s-g_s)$, \[\left|\int_0^T\int_{\mathbb{R}^d} D^s(\vartheta_s-g_s)\cdot\Phi\right|\leq C\norm{\Phi}_{L^2(0,T;L^2(\mathbb{R}^d)^d)},\] which implies, in the limit, that \[\left|\int_0^T\int_\Omega (\vartheta-g)(D\cdot\Phi)\right|=\left|\int_0^T\int_{\mathbb{R}^d} (\vartheta-g)\widetilde{(D\cdot\Phi)}\right|\leq C\norm{\Phi}_{L^2(0,T;L^2(\mathbb{R}^d)^d)}.\] Therefore we have $D(\vartheta-g)\in L^2(0,T;[L^2(\Omega)]^d)$
and hence
\[-\int_0^T\int_{\mathbb{R}^d} (\vartheta-g)\widetilde{(D\cdot\Phi)}=\int_0^T\int_{\mathbb{R}^d} D(\vartheta-g)\cdot\tilde{\Phi} \]
so $\zeta=D(\vartheta-g)$. Moreover, since $\vartheta-g=w-\lim_{s\nearrow1}(\vartheta_s-g_s)=0$ outside $\Omega$, and the boundary of $\Omega$ being Lipschitz, we may conclude $\vartheta-g\in L^\infty(0,T;H^1_0(\Omega))$.

We claim that $(\eta,\vartheta)$ satisfies the Stefan-type problem for $s=1$. Indeed, for any $\xi\in\Xi^1_T\subset \Xi^s_T$, \begin{multline*}-\int_{Q_T}\eta\frac{\partial \xi}{\partial t}+\int_0^T\int_\Omega AD(\vartheta-g)\cdot D\xi=-\int_{Q_T} \eta\frac{\partial \xi}{\partial t}+ \int_0^T\int_{\mathbb{R}^d}AD(\vartheta-g)\cdot \widetilde{D\xi}\\=\lim_{s\nearrow1}\left\{-\int_{Q_T} \eta_s \frac{\partial \xi}{\partial t}+\int_0^T\int_{\mathbb{R}^d}AD^s(\vartheta_s-g_s)\cdot D^s\xi\right\}=\lim_{s\nearrow1}\left\{\int_{Q_T} f\xi+\int_\Omega\eta_{0,s}\xi(0)\right\}=\int_{Q_T} f\xi+\int_\Omega\eta_0\xi(0)\end{multline*} since $D^s(\vartheta_s-g_s)\rightharpoonup D(\vartheta-g)$ in $L^\infty(0,T;L^2(\mathbb{R}^d)^d)$-weak$^*$, $\eta_s\rightharpoonup \eta$ in  $L^\infty(0,T;L^2(\Omega))$-weak$^*$, and $D^s\xi\to \widetilde{D\xi}$ strongly in $L^2(0,T;L^2(\mathbb{R}^d)^d)$ by Lemma \ref{sContDepStefan}. Therefore, $(\eta,\vartheta)$ satisfies \eqref{ClassicalStefanSLimit}.

Moreover, by Remark \ref{TempSBoundedRemark}, $\frac{\partial \vartheta_s}{\partial t}$ is bounded in $L^2(0,T;L^2(\Omega))$, so we can take the limit as $s\nearrow1$ to obtain that \[\frac{\partial \vartheta_s}{\partial t}\rightharpoonup\frac{\partial \vartheta}{\partial t}\text{ in }L^2(0,T;L^2(\mathbb{R}^d))\text{-weak}.\] Since $\frac{\partial g_s}{\partial t}\rightharpoonup \frac{\partial g}{\partial t}$ in $L^2(0,T;L^2(\mathbb{R}^d))$-weak, \[\vartheta_s-g_s\rightharpoonup\vartheta-g\text{ in } L^\infty(0,T;H^\sigma_0(\Omega))\text{-weak$^*$ and in } H^1(0,T;L^2(\Omega))\text{-weak}\] as $s\nearrow1$, and so by compactness (see, for instance, Corollary 4 of \cite{SimonCompactness}), \[\vartheta_s-g_s\to\vartheta-g\text{ in }C([0,T];L^2(\Omega)),\] giving the convergence \eqref{thetasconvg} as desired using the convergence of $g_s$ to $g$ in \eqref{sConvgAssumpConvg}.

Finally, it remains to show that $\vartheta=\gamma(\eta)$ a.e. in $\Omega\times]0,T[$, or equivalently $\eta\in\beta(\vartheta)$. Indeed, since $\vartheta_s=\gamma(\eta_s)$ a.e. in $\Omega\times]0,T[$ with $\eta_s\rightharpoonup \eta$ weakly in $L^2(0,T;L^2(\Omega))$ and $\vartheta_s\to\vartheta$ in $C([0,T];L^2(\Omega))$, by the maximal monotonicity of $\beta$ (see, for instance, Proposition 2.5 of \cite{BrezisBook1973}), we have $\eta\in\beta(\vartheta)$ and $\eta_0\in\beta(\vartheta(0))$ satisfying \eqref{ClassicalStefanSLimit}. Subsequently, we obtain the solution $\vartheta=\gamma(\eta)$ a.e. in $\Omega\times]0,T[$, with initial condition $\vartheta(0)=\lim_{s\nearrow1}\gamma(\eta_{0,s})=\gamma(\eta_0)$ by the convergence of $\eta_{0,s}$ to $\eta_0$ in $L^2(\Omega)$.

\end{proof}

\section{Asymptotic Behaviour as $t\to\infty$}\label{Sect:AsympBehavT}

In this section, we derive the asymptotic behaviour of the weak solutions as $t\to\infty$, following the approach of the classical case in \cite{DamlamianKenmochi1986Asymptotic}. We first begin with a well-known asymptotic convergence result for the solutions of differential equations with maximal monotone operators.

\begin{proposition}[See, for instance, Theorem 3.11 of \cite{BrezisBook1973}]\label{BrezisThmAsyptotic} Let $\varphi$ be a lower semi-continuous convex functional on a Hilbert space $H$. Suppose that for all $C\in\mathbb{R}$, the set $\{x\in H:\varphi(x)+|x|^2\leq C\}$ is compact. Let $f_\infty\in H$ and let $f(t)$ be a function such that $f-f_\infty\in L^1(t_0,\infty;H)$. Suppose $U\in C(t_0,\infty;H)$ is a weak solution to the equation $\frac{dU}{dt}+\partial\varphi(U)\ni f$. Then $\lim_{t\to+\infty} U(t)=U_\infty$ in $H$ exists and $f_\infty\in\partial\varphi(U_\infty)$.
\end{proposition} 

With this proposition, we can directly obtain the convergence of the generalised enthalpy solutions $\eta(t)\to\eta_\infty$ in the case where $\tilde{g}(t)=\tilde{g}_\infty$ for all $t\geq t_0$, i.e. the Dirichlet data is independent of time, with $f(t)-f_\infty\in L^1(t_0,\infty;H^{-s}(\Omega))$. For more general $\tilde{g}(t)$ converging to some $\tilde{g}_\infty$, we may also have a characterisation of the asymptotic behaviour of the generalised enthalpy solution towards the stationary solution, which can be written in terms of the stationary Dirichlet problem $\vartheta_\infty=g_\infty$ in $\Omega^c$ for the temperature $\vartheta_\infty$: 
\begin{equation}\label{FracStefanStationary}\int_{\mathbb{R}^d}AD^s\vartheta_\infty\cdot D^s\xi=\left\langle f_\infty,\xi\right\rangle,\quad\forall \xi\in H^s_0(\Omega).\end{equation}

\begin{theorem}\label{EnthalpyAsymptotic}
Let $f$, $\tilde{g}$ and $\eta_0$ satisfy the assumptions in Theorem \ref{VarStefanThm1} such that $f-f_\infty\in L^1(0,\infty;H^{-s}(\Omega))\cap L^2(0,\infty;H^{-s}(\Omega))$ and $\tilde{g}-\tilde{g}_\infty\in W^{1,1}(0,\infty;L^2(\Omega))$ for given $f_\infty\in H^{-s}(\Omega)$ and $\tilde{g}_\infty\in H^s(\mathbb{R}^d)$. (We can subsequently define $g_\infty$ and $g(t)$ in the same spaces using \eqref{DirichletBdryCond} as explained in the Appendix \ref{Sect:DirBdryCondAp}.) Let $\eta\in\beta(\vartheta)$ be the generalised enthalpy solution to the fractional Stefan-type problem \eqref{FracStefanEta} for all $T>0$. 
Then, there exists an $\eta_\infty\in L^2(\Omega)$ such that \[\eta(t)\to\eta_\infty\text{ strongly in }H^{-s}(\Omega)\text{ and weakly in }L^2(\Omega) \text{ as }t\to\infty, \] where $\eta_\infty$ is such that $\vartheta_\infty=\gamma(\eta_\infty)$ satisfies \eqref{FracStefanStationary} with $\vartheta_\infty=g_\infty$ in $\Omega^c$. 
\end{theorem}

\begin{proof}
We first note that, while $\eta_\infty$ is not unique in general, there exists a unique weak temperature solution $\vartheta_\infty=\gamma(\eta_\infty)$ to \eqref{FracStefanStationary} with $\vartheta_\infty=g_\infty$ in $\Omega^c$ by the Riesz representation theorem for $A$ coercive and bounded, since we have the equivalent norms \eqref{Hs0EquivNorm} in $H^s_0(\Omega)$.

Furthermore, under our assumptions, by a similar approach to the Proposition 3.2 and its Corollary in \cite{DamlamianKenmochi1986Asymptotic}, there is a positive constant $M$ such that \begin{equation}\label{EnthalpyRegInf}\sup_{t\geq0}\norm{\eta(t)}_{L^2(\Omega)}\leq M.\end{equation}

Let $\epsilon$ be any positive number. Since $g$ is bounded, we can take a number $t_\epsilon$ such that \[\int_{t_\epsilon}^\infty\norm{g(\tau)-g_\infty}_{L^2(\Omega)}+ \norm{f(\tau)-f_\infty}_{H^{-s}(\Omega)}\,d\tau\leq\epsilon.\] Also, let $w_\epsilon$ be the solution of the fractional Stefan-type problem \eqref{FracStefanEta} corresponding to $(f_\infty,g_\infty)$ on $[t_\epsilon,\infty[$ with initial value $w_\epsilon(t_\epsilon)=\eta(t_\epsilon)$, i.e. \begin{equation}\frac{d w_\epsilon}{dt}(t)+\partial\phi_\infty(w_\epsilon)=\frac{d w_\epsilon}{dt}(t)+\mathcal{L}_A^s(\gamma(w_\epsilon(t))-g_\infty)=f_\infty.\end{equation} By Proposition \ref{BrezisThmAsyptotic} in the interval $[t_\epsilon,\infty[$ with $H=H^{-s}(\Omega)$ and $\varphi$ given by the convex functional $\phi_\infty$ as defined in \eqref{ConvexFunctional} for the Dirichlet boundary condition $g_\infty$, since the set $\{W\in H^{-s}(\Omega):\phi_\infty(W)+|W|^2\leq C\}$ is a bounded set in $L^2$ and therefore compact in $H^{-s}(\Omega)$, we have that \[w_\epsilon(t) \text{ converges in }H^{-s}(\Omega)\text{ as }t\to\infty\text{ to a point }w_\epsilon^\infty\in L^2(\Omega)\] satisfying \begin{equation}\label{wepsilonStatProb}
f_\infty\in\partial\phi_\infty(w_\epsilon^\infty), \quad \text{or equivalently} \quad \mathcal{L}_A^s(\gamma(w_\epsilon^\infty)-g_\infty)=f_\infty.\end{equation} 
Therefore, there is a number $t'_\epsilon\geq t_\epsilon$ such that \[\norm{w_\epsilon(t)-w_\epsilon(\tau)}_{H^{-s}(\Omega)}\leq\epsilon\quad\forall t,\tau\geq t'_\epsilon.\] Also, as in Remark \ref{EnthalpyContDepGeneral}, we have that \begin{multline*}\frac{1}{2}\frac{d}{d\tau}\norm{\eta(\tau)-w_\epsilon(\tau)}_{H^{-s}(\Omega)}^2+\frac{2}{C_\gamma}\norm{\gamma(\eta)(\tau)-\gamma(w_\epsilon)(\tau)}_{L^2(\Omega)}^2\\\leq2\left( f(\tau)-f_\infty,\eta(\tau)-w_\epsilon(\tau)\right)+2\int_\Omega (g(\tau)-g_\infty)(\eta(\tau)-w_\epsilon(\tau)),\end{multline*} so in particular, \[\frac{d}{d\tau}\norm{\eta(\tau)-w_\epsilon(\tau)}_{H^{-s}(\Omega)}^2\leq K\left(\norm{f(\tau)-f_\infty}_{H^{-s}(\Omega)}+\norm{g(\tau)-g_\infty}_{L^2(\Omega)}\right)\] for some constant $K$ for a.e. $\tau\geq t_\epsilon$. 
Integrating both sides over $[t_\epsilon,t]$, we have \begin{equation}\label{AsymptoticHigherRegEst}\norm{\eta(t)-w_\epsilon(t)}_{H^{-s}(\Omega)}^2\leq K\epsilon\end{equation} for any $t\geq t_\epsilon$. Therefore, if $t,s\geq t'_\epsilon$, \[\norm{\eta(t)-\eta(s)}_{H^{-s}(\Omega)}\leq \norm{\eta(t)-w_\epsilon(t)}_{H^{-s}(\Omega)}+\norm{w_\epsilon(t)-w_\epsilon(s)}_{H^{-s}(\Omega)}+\norm{w_\epsilon(s)-\eta(s)}_{H^{-s}(\Omega)}\leq 2\sqrt{K\epsilon}+\epsilon.\] This implies that $\eta(t)$ converges in $H^{-s}(\Omega)$ as $t\to\infty$ to some $\eta_\infty\in H^{-s}(\Omega)$. Also, since \eqref{AsymptoticHigherRegEst} holds for all $t\geq t_\epsilon$ and $\lim_{t\to\infty}w_\epsilon(t)=w_\epsilon^\infty$, we have that $w_\epsilon^\infty\to \eta_\infty$ in $H^{-s}(\Omega)$ as $\epsilon\searrow0$. Since $w_\epsilon^\infty$ satisfies \eqref{wepsilonStatProb}, so does $\eta_\infty$.

Finally, defining $\vartheta_\infty=\gamma(\eta_\infty)$, taking the limit in $\epsilon$ in \eqref{wepsilonStatProb}, we have $\vartheta_\infty=(\mathcal{L}_A^s)^{-1}f_\infty+g_\infty$.
\end{proof}

\begin{remark}\label{EnthalpyRegAsymp}
In addition, if we assume $f-f_\infty\in W^{1,1}(0,\infty;H^{-s}(\Omega))$, we have that the solution $\eta$ to the fractional Stefan-type problem \eqref{FracStefanEta} satisfies $\eta-\eta_\infty\in H^1(0,\infty,H^{-s}(\Omega))$. This follows as in the proof of Theorem 2.1 of \cite{DamlamianKenmochi1986Asymptotic}, and it can be shown that the energy functional $J(t)$ given by \[J(t):=\phi_t(\eta(t))+\int_0^t\norm{\frac{d\eta}{d t}(\tau)}_{H^{-s}(\Omega)}^2\,d\tau-C\int_0^t\left(\norm{\frac{\partial g}{\partial t}(\tau)}_{L^2(\Omega)}+\norm{\frac{df}{dt}(\tau)}_{H^{-s}(\Omega)}\right)\,d\tau\quad\text{ for }t\geq0\] is bounded and non-increasing on $]0,\infty[$.  So $\lim_{t\to\infty}J(t)$ exists and 
\begin{equation}\label{EnthalpyRegRemarkEq}\frac{d\eta}{d t}\in L^2(0,\infty,H^{-s}(\Omega)).\end{equation}
\end{remark}

We can also increase the regularity of $\tilde{g}$ as in Theorem \ref{VarStefanThm2HigherReg} to obtain the convergence of $\vartheta$.

\begin{theorem}\label{TempAsymptotic}
Let $f-f_\infty\in L^1(0,\infty;H^{-s}(\Omega))\cap L^2(0,\infty;L^2(\Omega))$ and $\tilde{g}-\tilde{g}_\infty\in W^{2,1}(0,\infty;L^2(\mathbb{R}^d))\cap H^1(0,\infty;L^2(\mathbb{R}^d))\cap L^2(0,\infty;H^s(\mathbb{R}^d))$ (and so similarly with $g-g_\infty$), and $\eta_0\in L^2(\Omega)$, where $f_\infty\in L^2(\Omega)$ and $\tilde{g}_\infty\in H^s(\mathbb{R}^d)$. Suppose that $\vartheta$ is the weak temperature solution to the fractional Stefan-type problem \eqref{FracStefan}, and $\vartheta_\infty$ is the stationary weak temperature solution to \eqref{FracStefanStationary} with $\vartheta_\infty=g_\infty$ in $\Omega^c$. Then \[\vartheta(t)\to\vartheta_\infty\text{ in }L^2(\Omega) \text{ and in }H^s(\mathbb{R}^d)\text{-weak}\text{ as }t\to\infty.\] 

In addition, if $f-f_\infty\in W^{1,1}(0,\infty;H^{-s}(\Omega))$, we have \[\vartheta(t)-g(t)\to\vartheta_\infty-g_\infty\text{ strongly in }H^s_0(\Omega)\text{ as }t\to\infty.\] In particular, if $g(t)\to g_\infty$ in $H^s(\mathbb{R}^d)$ as $t\to\infty$, then $\vartheta(t)\to\vartheta_\infty$ strongly in $H^s(\mathbb{R}^d)$ as $t\to\infty$.
\end{theorem}

\begin{proof} 
Let $(\eta,\vartheta)$ be the solution to the fractional Stefan-type problem \eqref{FracStefan}, so that their finite-dimensional approximations $(\eta_n,\gamma(\eta_n))$ satisfy the inequality \eqref{A1CauchyProbSeqVar5}. Since the $\eta_n$'s are uniformly bounded in $L^\infty(0,\infty;L^2(\Omega))$ by Theorem \ref{EnthalpyAsymptotic} applied to the approximated problem, we have  \[\left|\int_0^\infty\int_\Omega \tilde{\eta}_n\frac{\partial^2 g}{\partial t^2}\right|\leq\norm{\tilde{\eta}_n}_{L^\infty(0,\infty;L^2(\Omega))}\norm{\frac{\partial^2 g}{\partial t^2}}_{L^1(0,\infty;L^2(\Omega))},\]\[\lim_{t\to\infty} \int_\Omega \tilde{\eta}_n(t)\frac{\partial g}{\partial t}(t)=0\quad\text{ since }\frac{\partial g}{\partial t}\to0\text{ in }L^2(\Omega),\]\[\left|\int_\Omega \tilde{\eta}_n(0)\frac{\partial g}{\partial t}(0)\right|\leq\norm{\tilde{\eta}_n(0)}_{L^2(\Omega)}\norm{\frac{\partial g}{\partial t}(0)}_{L^2(\Omega)},\] and \begin{multline*}\left|\int_0^\infty\int_\Omega f_n\frac{\partial g}{\partial t}\right|\leq\norm{f_n}_{L^2(0,\infty;L^2(\Omega))}\norm{\frac{\partial g}{\partial t}}_{L^2(0,\infty;L^2(\Omega))}\\=\norm{\mathbb{P}_{E_n}f}_{L^2(0,\infty;L^2(\Omega))}\norm{\frac{\partial g}{\partial t}}_{L^2(0,\infty;L^2(\Omega))}\leq\frac{a^*}{a_*}\norm{f}_{L^2(0,\infty;L^2(\Omega))}\norm{\frac{\partial g}{\partial t}}_{L^2(0,\infty;L^2(\Omega))}\end{multline*} and, passing to the limit in $n$ in \eqref{A1CauchyProbSeqVar5}, we conclude \begin{equation}\label{TempRegInf}\vartheta-g\in L^\infty(0,\infty;H^s_0(\Omega))\cap H^1(0,\infty;L^2(\Omega)).\end{equation}

Let $w^*$ be any accumulation point of $\{\vartheta(t)-g(t)\}$ in $H^s_0(\Omega)$ for the weak topology as $t\to\infty$, and let $\{t_n\}_n$ be a sequence in $[0,\infty[$ such that $t_n\nearrow\infty$ and $\vartheta(t_n)-g(t_n)\rightharpoonup w^*$ weakly in $H^s_0(\Omega)$ as $n\to\infty$. Then, by the convergence of $g$ and the compactness of $H^s_0(\Omega)$ in $L^2(\Omega)$, \[\vartheta(t_n)\to w^*+g_\infty\text{ in }L^2(\Omega).\] Also, from Theorem \ref{EnthalpyAsymptotic}, there exists an $\eta_\infty$ such that \[\eta(t_n)\rightharpoonup \eta_\infty\text{ in }L^2(\Omega)\text{-weak}.\] As $\vartheta(t_n)=\gamma(\eta(t_n))$, by the property of maximal monotone operators in $L^2(\Omega)$, the limit of any subsequence as $t_n\to\infty$ satisfies \[w^*+g_\infty=\gamma(\eta_\infty)=\vartheta_\infty.\] 
Therefore, $w^*=\vartheta_\infty-g_\infty$, and we have the convergence \begin{equation}\label{AsymConvTempL2}\vartheta(t)\to \vartheta_\infty\text{ in }L^2(\Omega)\text{ as }t\to\infty\end{equation} and \begin{equation}\label{AsymConvTempWeak}\vartheta(t)-g(t)\rightharpoonup \vartheta_\infty-g_\infty\text{ in }H^s_0(\Omega)\text{-weak}\text{ as }t\to\infty.\end{equation}

In order to obtain the strong convergence in \eqref{AsymConvTempWeak}, we define the function $E(t)$ by \begin{equation}\label{EnergyFuncEAsymp}E(t):=\frac{1}{C_\gamma}\int_0^t\norm{\frac{\partial\vartheta(\tau)}{\partial t}}_{L^2(\Omega)}^2\,d\tau+\frac{1}{2}\langle\mathcal{L}_A^s(\vartheta(t)-g(t)),\vartheta(t)-g(t)\rangle-\int_\Omega\eta(t)\frac{\partial g(t)}{\partial t}-\left\langle f(t),\vartheta(t)-g(t)\right\rangle\end{equation} for $t\geq0$. Then, using again the inequality \eqref{A1CauchyProbSeqVar5} in the limit $n\to\infty$ with the integral taken over the interval $[t_1,t_2]$ and incorporating the Lipschitz property in \eqref{LipschitzTimeDef}, we obtain
\begin{multline*}\frac{1}{C_\gamma}\int_{t_1}^{t_2} \norm{\frac{\partial \vartheta(\tau)}{\partial t}}_{L^2(\Omega)}^2\,d\tau+\frac{1}{2}\langle\mathcal{L}_A^s(\vartheta(t_2)-g(t_2)),\vartheta(t_2)-g(t_2)\rangle\\-\int_\Omega \eta(t_2)\frac{\partial g}{\partial t}(t_2)-\left\langle f(t_2) ,\vartheta(t_2)-g(t_2)\right\rangle\\\leq\frac{1}{2}\langle\mathcal{L}_A^s(\vartheta(t_1)-g(t_1)),\vartheta(t_1)-g(t_1)\rangle-\int_\Omega\eta(t_1)\frac{\partial g}{\partial t}(t_1)-\left\langle f(t_1) ,\vartheta(t_1)-g(t_1)\right\rangle\\-\int_{t_1}^{t_2} \left\langle\frac{\partial f(\tau)}{\partial t} ,\vartheta(\tau)-g(\tau)\right\rangle\,d\tau-\int_{t_1}^{t_2}\int_\Omega\eta(\tau)\frac{\partial^2 g(\tau)}{\partial t^2}\,d\tau,\end{multline*} or \begin{equation}E(t_2)\leq E(t_1)-\int_{t_1}^{t_2} \left\{\left\langle\frac{\partial f(\tau)}{\partial t} ,\vartheta(\tau)-g(\tau)\right\rangle+\int_\Omega\eta(\tau)\frac{\partial^2 g(\tau)}{\partial t^2}\right\}\,d\tau.\end{equation} 
Recalling  \eqref{EnthalpyRegInf} and \eqref{TempRegInf}, we have $\eta\in L^\infty(0,T;L^2(\Omega))$ and $\vartheta-g\in L^\infty(0,T;H^s_0(\Omega))$, and so \[\int_{t_1}^{t_2}\left( \left\langle \frac{\partial f(\tau)}{\partial t} ,\vartheta(\tau)-g(\tau)\right\rangle+\int_\Omega\eta(\tau)\frac{\partial^2 g(\tau)}{\partial t^2}\right)\,d\tau\leq K_1\int_{t_1}^{t_2} \norm{\frac{\partial f(\tau)}{\partial t}}_{H^{-s}(\Omega)}\,d\tau+K_2\int_{t_1}^{t_2}\norm{\frac{\partial^2 g(\tau)}{\partial t^2}}_{L^2(\Omega)}\,d\tau\] for some constants $K_1,K_2\geq0$ for any $t_2\geq t_1\geq0$. 
Setting $H$ to be the function \[H(\cdot):=K_1\norm{\frac{\partial f(\cdot)}{\partial t}}_{H^{-s}(\Omega)}+K_2\norm{\frac{\partial^2 g(\cdot)}{\partial t^2}}_{L^2(\Omega)}\in L^1(0,\infty),\] it follows that \[E(t_2)-\int_0^{t_2}H(\tau)\,d\tau\leq E(t_1)-\int_0^{t_1}H(\tau)\,d\tau\quad\text{ for all }t_2\geq t_1\geq0.\] 
This implies that $\lim_{t\to\infty}E(t)$ exists, which we write as $E_\infty$ and, by definition \eqref{EnergyFuncEAsymp}, \begin{equation}\label{AsympLimitTemp}\lim_{t\to\infty}\langle\mathcal{L}_A^s(\vartheta(t)-g(t)),\vartheta(t)-g(t)\rangle=2E_\infty-\frac{2}{C_\gamma}\int_0^\infty\norm{\frac{\partial\vartheta(\tau)}{\partial t}}_{L^2(\Omega)}^2\,d\tau+2\left\langle f_\infty,\vartheta_\infty-g_\infty\right\rangle=:l_\infty\end{equation} since $\eta$ is bounded in $L^2(\Omega)$ and $\frac{\partial g(t)}{\partial t}\to0$ in $L^2(\mathbb{R}^d)$ as $t\to\infty$. 

Next, taking a sequence $\{t_n\}_n$ with $t_n\to\infty$ so that \[\frac{d\eta}{dt}(t_n)\to0\text{ in }H^{-s}(\Omega),\] which is always possible by \eqref{EnthalpyRegRemarkEq}, we have, recalling that $\vartheta_\infty$ is the weak temperature solution to \eqref{FracStefanStationary}, \begin{equation}\label{AsymConv}\mathcal{L}_A^s(\vartheta(t_n)-g(t_n))=f(t_n)-\frac{d\eta}{dt}(t_n)\to f_\infty=\mathcal{L}_A^s(\vartheta_\infty-g_\infty)\text{ in }H^{-s}(\Omega).\end{equation} 
Therefore, by \eqref{AsymConvTempWeak}, \eqref{AsymConv} and \eqref{AsympLimitTemp}, \begin{equation}\label{AsymConvDuality}\langle\mathcal{L}_A^s(\vartheta(t_n)-g(t_n)),\vartheta(t_n)-g(t_n)\rangle\to\langle\mathcal{L}_A^s(\vartheta_\infty-g_\infty),\vartheta_\infty-g_\infty\rangle=l_\infty.\end{equation} Finally, since the duality in the left hand side of \eqref{AsympLimitTemp} is equivalent to the square of the $H^s_0(\Omega)$ norm of $\vartheta(t)-g(t)$ by \eqref{Fdef}, we may conclude the strong convergence result \[\vartheta(t)-g(t)\to\vartheta_\infty-g_\infty\text{ in }H^s_0(\Omega)\text{ as }t\to\infty.\]

\end{proof}

\begin{remark}
Since $\eta(t)=b(\vartheta(t))+\chi(t)$, $\chi(t)\in H(\vartheta(t))$, and $\eta(t)\xrightharpoonup[t\to\infty]{} \eta_\infty$ in $L^2(\Omega)$-weak and $\vartheta(t)\to\vartheta_\infty$ in $L^2(\Omega)$, we have the existence of a $\chi_\infty\in H(\vartheta_\infty)$, such that $\chi(t)\xrightharpoonup[t\to\infty]{}\chi_\infty$ in $L^\infty(\Omega)$-weak$^*$.
\end{remark}

\begin{remark}
Similar asymptotic results as $t\to\infty$ for the case $s=1$ have been obtained in \cite{DamlamianKenmochi1986Asymptotic} considering other variants on the asymptotic behaviour of $f$ and $\tilde{g}$. 

Earlier asymptotic behaviour results for $s=1$ were obtained in Remarks 9 and 11 of \cite{TarziaDuvaut1} in the variational inequality form in a special case.
\end{remark}

\section{From Two Phases to One Phase}\label{Sect:OnePhase}
Let $\nu$ be a parameter such that \eqref{FracStefan} written with the Lipschitz graph $\gamma^\nu$ corresponds to the two-phase problem when $\nu>0$, and to the one-phase problem when $\nu=0$.
In this section, we obtain the solution to the one-phase problem, making use of the solution to the two-phase problem. 

Consider the one-phase problem given with data $f^\omicron $, $\tilde{g}^\omicron \geq0$ by \begin{equation}\tag{\ref{FracStefan}$_{1ph}$}-\int_{Q_T}\eta^\omicron \frac{\partial \xi}{\partial t}+\int_{\mathbb{R}^d}AD^s\vartheta^\omicron \cdot D^s\xi=\int_{Q_T}f^\omicron \xi+\int_\Omega\eta^\omicron _0\xi(0),\quad\forall \xi\in \Xi^s_T\end{equation} with initial condition $\eta^\omicron (x,0)=\eta_0^\omicron (x)$ with regularity as in Theorem \ref{VarStefanThm2HigherReg} and $\vartheta^\omicron =\gamma^\omicron (\eta^\omicron )$ such that $\vartheta^\omicron (0)-g^\omicron (0)\in H^s_0(\Omega)$. In this section, we use the lower subscript $\omicron$ to indicate the one-phase problem, and the upper superscript 0 to indicate the initial condition. We first show that there exists a solution to this problem, by obtaining the solution as the limit of a sequence of solutions to two-phase problems. The main idea is that we flatten the left leg of the monotone Lipschitz graph $\gamma$ to obtain $\gamma^\omicron $ which has range $[0,\infty[$. Then $\gamma^\omicron $ will still satisfy the same conditions \eqref{gammacond} at $r=+\infty$. Furthermore, we define the convex functional $\phi_t^\omicron$ by \[\phi_t^\omicron(W)=\begin{cases}\int_\Omega(j^\omicron (W)-g^\omicron (t)W)\,dx&\text{ for }W\in L^2(\Omega);\\+\infty &\text{ for }W\in H^{-s}(\Omega)\backslash L^2(\Omega)\end{cases}\] for the primitive $j^\omicron $ of $\gamma^\omicron $ chosen such that $j^\omicron $ vanishes at 0.

\begin{remark} Observe that the image of $\gamma^\omicron $ is $[0,\infty[$. Therefore, given any $\eta_0^\omicron \in L^2(\Omega)$, $\vartheta^\omicron (0)=\gamma^\omicron (\eta_0^\omicron )\geq0$. This also applies to $\eta^\omicron (t)\in L^2(\Omega)$ at general time $t\in[0,T]$, so we have $\vartheta^\omicron (t)=\gamma^\omicron (\eta^\omicron (t))\geq0$ for all $t$. As such, it is necessary that the Dirichlet boundary condition $g^\omicron $ is non-negative in $\Omega^c\times]0,T[$.
\end{remark}

\begin{theorem}\label{1phExistence}
Let $f^\omicron \in L^2(0,T;L^2(\Omega))$ and $\tilde{g}^\omicron \in W^{2,1}(0,T;L^2(\mathbb{R}^d))\cap L^\infty(0,T;H^s(\mathbb{R}^d))$, and define $g^\omicron $ as in \eqref{DirichletBdryCond} (and subsequently with the same regularity). Assume $\eta_0^\omicron \in L^2(\Omega)$ and, setting $0\leq\vartheta^\omicron (0)=\gamma^\omicron (\eta_0^\omicron )$, assume $\tilde{g}^\omicron \geq0$ in $\Omega^c\times]0,T[$ and $\vartheta^\omicron (0)-g^\omicron (0)\in H^s_0(\Omega)$. Then, there exist a unique generalised enthalpy solution $\eta^\omicron $ and a weak temperature solution $\vartheta^\omicron $ to the variational problem (\ref{FracStefan}$_{1ph}$) with \[\eta^\omicron \in\beta^\omicron (\vartheta^\omicron )\quad\text{ and }\quad\vartheta^\omicron =\gamma^\omicron (\eta^\omicron )\geq0,\] such that \begin{equation}\label{1phRegEta}\eta^\omicron \in L^\infty(0,T;L^2(\Omega))\cap H^1(0,T;H^{-s}(\Omega))\end{equation} and \begin{equation}\label{1phRegTheta} \vartheta^\omicron \in L^\infty(0,T;H^s(\mathbb{R}^d))\cap H^1(0,T;L^2(\Omega))\end{equation} with $\vartheta^\omicron =g^\omicron $ in $\Omega^c$.
\end{theorem}

\begin{proof}
We construct $\eta^\omicron $ and $\vartheta^\omicron $ as the limit of an approximating sequence of $\eta^\nu$ and $\vartheta^\nu$. (See also the proof of Theorem A.1 in \cite{DelTesoVazquez1Phase}.)

Indeed, since $\gamma^\omicron $ is non-negative, \[\lim_{|r|\to+\infty}\frac{\gamma^\omicron (r)}{r}\geq0.\] Then, consider the strictly increasing approximation \begin{equation}\gamma^\nu(r)=\gamma^\omicron (r)+\nu r\end{equation} for $\nu>0$. Assuming $\gamma^\omicron $ is Lipschitz continuous, so is $\gamma^\nu$. Also, $\gamma^\nu$ clearly converges to $\gamma^\omicron $ uniformly on compact sets as $\nu$ tends to zero. Furthermore,  \[\liminf_{|r|\to+\infty}\frac{\gamma^\nu(r)}{r}\geq\nu+\liminf_{|r|\to+\infty}\frac{\gamma^\omicron (r)}{r}>0,\] so \eqref{gammacond} is satisfied. The corresponding maximal monotone graph $\beta^\nu$ is then given by \begin{equation}\beta^\nu(r)=\frac{1}{\nu}(r-(Id+\nu\beta^\omicron )^{-1}(r)),\end{equation} which is Lipschitz continuous with constant $\frac{1}{\nu}$. Therefore, from Theorem \ref{VarStefanThm1} and Theorem \ref{VarStefanThm2HigherReg}, we obtain the unique generalised enthalpy and weak temperature solutions $\eta^\nu$ and $\vartheta^\nu$ of the approximate regularized problem with approximating compatible functions $f^\nu$, $g^\nu$ and $\eta_0^\nu$ in the same spaces as the ones of the data \begin{equation}\tag{\ref{FracStefan}$^\nu$}-\int_{Q_T}\eta^\nu \frac{\partial\xi}{\partial t}+\int_{\mathbb{R}^d\times[0,T]}AD^s\vartheta^\nu\cdot D^s\xi=\int_{Q_T} f^\nu\xi+\int_\Omega\eta^\nu_0\xi(0),\quad\forall \xi\in \Xi^s_T,\end{equation} such that $\eta^\nu=\beta^\nu(\vartheta^\nu)$ are uniformly bounded in $H^1(0,T;H^{-s}(\Omega)\cap L^\infty(0,T;L^2(\Omega))$ for $\nu<1$, since the estimates \eqref{EstEq1}--\eqref{EstEq2} are independent of $\nu$ with \begin{align*}\phi_{t,n}^\nu(\eta^\nu_0)&=\int_\Omega(j^\nu(\eta^\nu_0)+g^\nu(0)\eta^\nu_0)+I_{F_n^*}\\&=\int_\Omega(j^\omicron (\eta^\nu_0)+\nu|\eta^\nu_0|^2+g^\nu(0)\eta^\nu_0)+I_{F_n^*}\\&\leq\int_\Omega(j^\omicron (\eta^\nu_0)+|\eta^\nu_0|^2+g^\nu(0)\eta^\nu_0)+I_{F_n^*}\end{align*} for uniformly bounded $\eta^\nu_0,g^\nu(0)\in L^2(\Omega)$. We recall that $(F_n)_{n\in\mathbb{N}}$ is an increasing set of finite dimensional subspaces of $H^s_0(\Omega)$, $F_n^*=\mathcal{L}(F_n)\subset H^{-s}(\Omega)$, and $I_{F_n^*}$ is the indicator function of $F_n^*$, i.e. $I_{F_n^*}=0$ in $F_n^*$, $I_{F_n^*}=+\infty$ elsewhere.

Henceforth, taking $C_\gamma=C_{\gamma^\omicron }+1$ in \eqref{LipschitzTimeDef} and making use of \eqref{A1CauchyProbSeqVar5} at the limit $n\to\infty$, we obtain that $\frac{\partial  \vartheta^\nu}{\partial t}$ is bounded in $L^2(0,T;L^2(\Omega))$ and $\vartheta^\nu-g^\nu$ is bounded in $L^\infty(0,T;H^s_0(\Omega))$ independently of $\nu$.
Passing to the limit as $\nu$ tends to zero, since $\eta^\nu$ is bounded in $H^1(0,T;H^{-s}(\Omega))$ as a solution to (\ref{FracStefan}$_{\nu}$), we have $(\eta^{\nu_n})_n$ converging in $H^1(0,T;H^{-s}(\Omega))$-weak and in $L^\infty(0,T;L^2(\Omega))$ in the weak$^*$ topology, to some $\eta^\omicron $. Similarly, $(\vartheta^{\nu_n})_n=\left(\gamma^{\nu_n}(\eta^{\nu_n})\right)_n$ converges weakly in $H^1(0,T;L^2(\Omega))\cap L^\infty(0,T;H^s(\mathbb{R}^d))$, and by compactness also in $C([0,T];L^2(\Omega))$, to some $\vartheta^\omicron $ such that $\vartheta^\omicron (t)-g^\omicron (t)\in H^s_0(\Omega)$ a.e. $t$. Passing to the limit, $\vartheta^\omicron $ satisfies (\ref{FracStefan}$_{1ph}$) with the required regularity \eqref{1phRegTheta}. Also, by the maximal monotonicity of $\beta^\omicron $ and the Mosco convergence of $\beta^\nu$ to $\beta^\omicron $, we have $\eta^\omicron \in\beta^\omicron (\vartheta^\omicron )$ and $\eta^\omicron _0\in\beta^\omicron (\vartheta^\omicron (0))$ satisfying (\ref{FracStefan}$_{1ph}$) and \eqref{1phRegEta}. Subsequently, $\vartheta^\omicron =\gamma^\omicron (\eta^\omicron )$ a.e. in $\Omega\times]0,T[$ and $\vartheta^\omicron (0)=\lim_{\nu\to0}\gamma^\nu(\eta_0^\nu)=\gamma^\omicron (\eta_0^\omicron )$ by the convergence of $\eta_0^\nu$ to $\eta_0^\omicron $ in $L^2(\Omega)$. Since the range of $\gamma^\omicron $ is $[0,\infty[$, $\vartheta\geq0$ and we obtain the solution of the one-phase problem.
\end{proof}

Having obtained a unique solution to the limiting one-phase problem, we now show that the solutions of the two-phase problem given by \begin{equation}-\eta^\nu \int_{Q_T} \frac{\partial\xi}{\partial t}+\int_{\mathbb{R}^d\times[0,T]}AD^s(\vartheta^\nu-g^\nu)\cdot D^s\xi=\int_{Q_T} f^\nu\xi+\int_\Omega\eta^\nu_0\xi(0)\quad\forall \xi\in \Xi^s_T\tag{\ref{FracStefan}$_{2ph}$}\end{equation} with $\vartheta^\nu=\gamma^\nu(\eta^\nu)$ in fact converges to the one-phase problem (\ref{FracStefan}$_{1ph}$). For the classical case of $s=1$, see also \cite{BarbaraStoth19972ph1ph}, as well as the proof of Theorem 6.1 on pages 44-45 of \cite{damlamian1976thesis}).

\sloppy \begin{theorem}\label{ConvergenceThm2ph} Assume that for each $\nu\geq0$, $f^\nu\in L^2(0,T;L^2(\Omega))$, $\tilde{g}^\nu\in W^{2,1}(0,T;L^2(\Omega))\cap L^\infty(0,T;H^s(\mathbb{R}^d))$ bounded independently of $\nu$, and $\eta_0^\nu\in \overline{D(\phi_t^\nu)}$. Writing $\vartheta^\nu=\gamma^\nu(\eta^\nu)$ for the Lipschitz graph $\gamma^\nu$ with a uniform Lipschitz constant $C_\gamma$ for all $\nu\geq0$, assume that $\eta_0^\nu\in L^2(\Omega)$ and, setting $0\leq\vartheta^\nu (0)=\gamma^\nu (\eta_0^\nu )$, assume $\tilde{g}^\nu \geq0$ in $\Omega^c\times]0,T[$ and $\vartheta^\nu (0)-g^\nu (0)\in H^s_0(\Omega)$ is bounded uniformly in $\nu$ for $\nu\geq0$. Let $(\eta^\nu,\vartheta^\nu)$ be the unique solution of the fractional two-phase Stefan-type problem (\ref{FracStefan}$_{2ph}$), while $(\eta^\omicron ,\vartheta^\omicron )$ is the unique solution of the fractional one-phase Stefan-type problem (\ref{FracStefan}$_{1ph}$) with $0\leq\vartheta^\omicron =\gamma^\omicron (\eta^\omicron )$. Suppose that $\eta^\nu_0\rightharpoonup\eta^\omicron _0$ in $L^2(\Omega)$, $f^\nu\rightharpoonup f^\omicron $ in $L^2(0,T;L^2(\Omega))$, $g^\nu\rightharpoonup g^\omicron $ in $W^{2,1}(0,T;L^2(\mathbb{R}^d))$-weak and in $L^\infty(0,T;H^s(\mathbb{R}^d))$-weak$^*$, and $\gamma^\nu$ converges to $\gamma^\omicron $ uniformly on compact sets as $\nu$ tends to zero. Then, \[\eta^\nu\rightharpoonup \eta^\omicron \text{ in } H^1(0,T;H^{-s}(\Omega))\text{-weak}\text{ and in }L^\infty(0,T;L^2(\Omega))\text{-weak}^*\text{ as }\nu\searrow0\] and \[\vartheta^\nu\rightharpoonup\vartheta^\omicron \text{ in }H^1(0,T;L^2(\Omega))\text{-weak,}\text{ in }L^\infty(0,T;H^s(\mathbb{R}^d))\text{-weak}^*\text{ and in }C([0,T];L^2(\Omega))\text{ as }\nu\searrow0.\]
\end{theorem}

\begin{proof}
Indeed, as in the previous theorem, since $\eta^\nu\in\beta^\nu(\vartheta^\nu)$ is a solution to (\ref{FracStefan}$_{2ph}$), it is bounded in $H^1(0,T;H^{-s}(\Omega))$. Passing to a subsequence, we have $(\eta^{\nu_n})_n$ converging in $H^1(0,T;H^{-s}(\Omega))$-weak and in $L^\infty(0,T;L^2(\Omega))$ in the weak$^*$ topology, to some $\eta^\omicron $. 

Furthermore, \[\mathcal{L}_A^s(\gamma^\nu(\eta^\nu)-g^\nu)=\partial\phi^\nu_t(\eta^\nu)=f^\nu-\frac{\partial \eta^\nu}{\partial t}\rightharpoonup f^\omicron -\frac{\partial \eta^\omicron }{\partial t}=\partial\phi^\omicron_t(\eta^\omicron )=\mathcal{L}_A^s(\gamma^\omicron (\eta^\omicron )-g^\omicron )\text{ weakly in }L^2(0,T;H^{-s}(\Omega).\] Therefore, by applying $(\mathcal{L}_A^s)^{-1}$, $w^\nu=\gamma^\nu(\eta^\nu)-g^\nu$ converges weakly to $w^\omicron =\gamma^\omicron (\eta^\omicron )-g^\omicron $ in $L^2(0,T;H^s_0(\Omega))$. But $\eta^\nu$ is in $L^\infty(0,T;L^2(\Omega))$ for each $\nu$ by Theorem \ref{VarStefanThm1} since $\eta^\nu$ is the generalised enthalpy solution to the Stefan-type problem (\ref{FracStefan}$_{2ph}$), bounded independent of $\nu>0$ for $\nu$ small enough. Therefore, by the assumptions, we can again obtain a priori estimates on $\vartheta_\nu=\gamma^\nu(\eta^\nu)$ in $L^\infty(0,T;H^s(\mathbb{R}^d))\cap H^1(0,T;L^2(\Omega))$, and the conclusion follows as in the proof of the previous theorem.
\end{proof}

\begin{remark}
Similarly to the convergence of the two-phase problem, it is possible to extend the results of Sections 4 and 5 to the one-phase problem. 
\end{remark}

\appendix
\appendix

\section{Appendix - The Fractional Dirichlet Problem}\label{Sect:DirBdryCondAp}

The function $g=g(t)$ is constructed for every fixed $t\in J$, for the interval $J=[0,T]$ for all $T<\infty$, (using Theorem 1.13 of \cite{SS1}) by solving \begin{equation}\label{A2Prob1}\int_{\mathbb{R}^d}AD^sg(t)\cdot D^sv=0\quad\forall v\in H^s_0(\Omega)\end{equation} 
with the Dirichlet boundary condition given by \[g(t)=\tilde{g}(t)\text{ in }\Omega^c,\]
with $\tilde{g}(t)$ defined on $H^s(\mathbb{R}^d)$. When $\tilde{g}\in BV(0,T;H^s(\mathbb{R}^d))$ or $H^k(0,T;H^s(\mathbb{R}^d))$ for $k=1,2$, 
by solving this Dirichlet problem, $g$ will have the same time regularity as $\tilde{g}$.

Indeed, consider $u=g-\tilde{g}$. Then $u$ satisfies $u(t)=0\text{ in }\Omega^c$ and \begin{equation}\label{A2Prob2}\int_{\mathbb{R}^d}AD^su(t)\cdot D^sv=-\int_{\mathbb{R}^d}AD^s\tilde{g}(t)\cdot D^sv=:\langle\mathcal{L}\tilde{g}(t),v\rangle\quad\forall v\in H^s_0(\Omega)\end{equation} Since $\mathcal{L}:H^s(\mathbb{R}^d)\to H^{-s}(\Omega)$ with $\tilde{g}(t)\in H^s(\mathbb{R}^d)$, $\mathcal{L}\tilde{g}(t)$ is a linear functional in $H^{-s}(\Omega)$. By the coercivity and boundedness of $\mathcal{L}$, there exists a unique solution $u(t)\in H^s_0(\Omega)$ satisfying \eqref{A2Prob2} for almost every $t\in J$ by the Lax-Milgram theorem. By the uniqueness of $u(t)$, there exists a unique $g(t):=u(t)+\tilde{g}(t)\in H^s(\mathbb{R}^d)$ satisfying \eqref{A2Prob1} for almost every $t\in J$. It is clear that $g\in L^2(0,T;H^s(\mathbb{R}^d))$ if $\tilde{g}\in L^2(0,T;H^s(\mathbb{R}^d))$.

Furthermore, by linearity of $\mathcal{L}$, considering two time slices $\{t\}\times\Omega$ and $\{\tau\}\times\Omega$, we have, taking the test function to be $u(t)-u(\tau)$, \begin{align}\label{A2timeEst1}\begin{split} a_*\norm{u(t)-u(\tau)}_{H^s_0(\Omega)}^2&\leq\int_{\mathbb{R}^d}AD^su(t)\cdot D^s(u(t)-u(\tau))-\int_{\mathbb{R}^d}AD^su(\tau)\cdot D^s(u(t)-u(\tau))\\&=-\int_{\mathbb{R}^d}AD^s\tilde{g}(t)\cdot D^s(u(t)-u(\tau))+\int_{\mathbb{R}^d}AD^s\tilde{g}(\tau)\cdot D^s(u(t)-u(\tau))\\&\leq a^*\norm{\tilde{g}(t)-\tilde{g}(\tau)}_{H^s(\mathbb{R}^d)}\norm{u(t)-u(\tau)}_{H^s_0(\Omega)},\end{split}\end{align} so taking the sum of all time steps in $[t_i,t_{i-1}]\subset[0,T]$, $u\in BV(0,T;H^s_0(\Omega))$ if $\tilde{g}\in BV(0,T;H^s(\mathbb{R}^d))$, and consequently $g=u+\tilde{g}\in BV(0,T;H^s(\mathbb{R}^d))$.

Also, from \eqref{A2timeEst1}, we have the continuity of $u(t)$ in time for $t\in J$. Therefore, $u\in C(J;H^s_0(\Omega))$ if $\tilde{g}(t)$ is continuous for $t\in J$. Furthermore, we consider the problem \begin{equation}\label{A2Prob3}\int_{\mathbb{R}^d}AD^sw(t)\cdot D^sv=-\int_{\mathbb{R}^d}AD^s\frac{\partial \tilde{g} }{\partial t}(t)\cdot D^sv=\left\langle\mathcal{L}\frac{\partial \tilde{g} }{\partial t}(t),v\right\rangle\quad\forall v\in H^s_0(\Omega)\end{equation} when $\frac{\partial \tilde{g} }{\partial t}\in H^s(\mathbb{R}^d)$, and we can once again apply the argument above to obtain a unique solution $w\in H^s_0(\Omega)$ for almost every $t\in J$. It remains to show that \[w(t)=\frac{\partial u}{\partial t}(t)\quad\text{ a.e. }t\text{ in }H^s_0(\Omega).\] But, as in \eqref{A2timeEst1}, we have, using \eqref{A2Prob2} and \eqref{A2Prob3} and taking the test function to be $\frac{u(t)-u(t+h)}{h}-w(t)$, \begin{align}\label{A2timeEst2}\begin{split} &\,a_*\norm{\frac{u(t)-u(t+h)}{h}-w(t)}_{H^s_0(\Omega)}^2\\\leq&\,\int_{\mathbb{R}^d}AD^s\frac{u(t)-u(t+h)}{h}\cdot D^s\left(\frac{u(t)-u(t+h)}{h}-w(t)\right)-\int_{\mathbb{R}^d}AD^sw(t)\cdot D^s\left(\frac{u(t)-u(t+h)}{h}-w(t)\right)\\=&\,-\int_{\mathbb{R}^d}AD^s\frac{\tilde{g}(t)-\tilde{g}(t+h)}{h}\cdot D^s\left(\frac{u(t)-u(t+h)}{h}-w(t)\right)+\int_{\mathbb{R}^d}AD^s\frac{\partial \tilde{g} }{\partial t}(t)\cdot D^s\left(\frac{u(t)-u(t+h)}{h}-w(t)\right)\\\leq&\, a^*\norm{\frac{\tilde{g}(t)-\tilde{g}(t+h)}{h}-\frac{\partial \tilde{g} }{\partial t}(t)}_{H^s(\mathbb{R}^d)}\norm{\frac{u(t)-u(t+h)}{h}-w(t)}_{H^s_0(\Omega)}.\end{split}\end{align} But recall that by definition (see, for instance, Chapter 23.5 of \cite{Zeidlerbook2A}), \[\frac{\tilde{g}(t)-\tilde{g}(t+h)}{h}\to\frac{\partial \tilde{g} }{\partial t}(t)\quad\text{ in }H^s(\mathbb{R}^d)\text{ as }h\to0.\] Therefore, for any $\epsilon>0$, take a small enough $h>0$ such that $\norm{\frac{\tilde{g}(t)-\tilde{g}(t+h)}{h}-\frac{\partial \tilde{g} }{\partial t}(t)}_{H^s(\mathbb{R}^d)}<\epsilon$, then $\norm{\frac{u(t)-u(t+h)}{h}-w(t)}_{H^s_0(\Omega)}<\frac{a^*\epsilon}{a_*}$. Since $\epsilon$ is arbitrary, \[w(t)=\lim_{h\to0}\frac{u(t)-u(t+h)}{h}\quad\text{ a.e. }t\text{ in }H^s_0(\Omega),\] and the limit of the difference quotient is, by definition, $\frac{\partial u}{\partial t}$.
Therefore, $\frac{\partial g}{\partial t}=w(t)+\frac{\partial \tilde{g} }{\partial t}(t)$, and we have that $g$ has the same regularity as $\tilde{g}$ in $H^1(0,T;H^s(\mathbb{R}^d))$.  Repeating this argument again by taking a second time derivative, we have the same result for $g$ if $\tilde{g}\in H^2(0,T;H^s(\mathbb{R}^d))$.

Analogously, for $\tilde{g}\in W^{2,1}(0,T;L^2(\mathbb{R}^d))\cap L^2(0,T;H^s(\mathbb{R}^d))$ for $T\in]0,\infty]$, $g$ is first constructed from $\tilde{g}\in H^2(0,T;H^s(\mathbb{R}^d))$, and then extended by density to obtain also $g\in W^{2,1}(0,T;L^2(\mathbb{R}^d))\cap L^2(0,T;H^s(\mathbb{R}^d))$.

\section{Appendix - The Variational Inequality Formulations}\label{Sect:VarIneqEquiv}
We observe that the formulation given in \eqref{FracStefan} can be formally transformed into a variational inequality formulation with fractional derivatives (see for example \cite{RodriguesVarMethods1994b} or Chapter VII of \cite{damlamian1976thesis}). 
Indeed, consider an element $w\in H^s_0(\Omega)$ independent of $t$ and taking in \eqref{FracStefanEta} the test function $\xi(x,\tau)=w(x)$ for $\tau\in]t-\epsilon,t+\epsilon[$ and $\xi(x,\tau)=0$, dividing by $2\epsilon$ and letting $\epsilon\to0$, denoting now by $\langle\cdot,\cdot\rangle$ the duality between $H^{-s}(\Omega)$ and $H^s_0(\Omega)$, we obtain \[\left\langle\frac{d\eta}{dt}(t),w\right\rangle+\langle \mathcal{L}^s_A(\gamma(\eta(t))-g(t)),w\rangle=\langle f(t),w\rangle\quad\text{ for a.e. }t\text{ for all }w\in H^s_0(\Omega).\]

Then, integrating with respect to time and using the regularity of $\eta$ and its initial condition, we have,  \begin{equation}\label{Equivalence0}\int_\Omega \eta(t)w+\int_0^t\int_{\mathbb{R}^d} AD^s(\vartheta)\cdot D^sw=\int_0^t\int_\Omega fw+\int_\Omega \eta_0w\end{equation} for almost all $t\in[0,T]$ and $w\in H^s_0(\Omega)$ by recalling that $\int_0^t\int_{\mathbb{R}^d} AD^sg\cdot D^sw=0$ for all $w$. 
We write $\eta(t)=b(\vartheta(t))+\lambda\chi(t)$ for a.e. $t$ for $\lambda>0$ and $b$ a given continuous and increasing function (see Figure \ref{fig:EnthalpyGraph}). Then, denoting
\[\Theta(t)=\int_0^t\vartheta(\tau)\,d\tau \quad \text{ and }\quad  \mathfrak{F}(t)=\int_0^tf(\tau)\,d\tau,\] we observe that $b(\vartheta(t))=b\left(\frac{\partial\Theta}{\partial t}(t)\right)\in L^2(\Omega)$ a.e. $t$. On the other hand, since $H(r)$ is the subdifferential of the convex function $r^+$, we have the inequality \begin{equation}\label{signfuncineq}s\chi\leq(r+s)^+-r^+.\end{equation} So, we obtain from \eqref{Equivalence0} the nonlocal variational inequality \begin{equation}\label{Equivalence2}\int_\Omega b\left(\frac{\partial\Theta}{\partial t}(t)\right)w+\int_{\mathbb{R}^d} AD^s\Theta(t)\cdot D^sw+\int_\Omega \lambda\left(\frac{\partial\Theta}{\partial t}(t)+w\right)^+\geq\int_\Omega \lambda\left(\frac{\partial\Theta}{\partial t}(t)\right)^++\int_\Omega (\mathfrak{F}(t)+\eta_0)w\end{equation} for all $w\in H^s_0(\Omega)$ for a.e. $t$. 

By Theorem \ref{VarStefanThm1}, $\vartheta-g\in L^2(0,T;H^s_0(\Omega))$, so $\Theta$ satisfies \begin{equation}\label{Equivalence1}\Theta\in H^1(0,T;H^s(\mathbb{R}^d)),\quad\Theta(0)=0,\quad\text{ and }\Theta(t)-\int_0^tg(\tau)\,d\tau=0 \text{ in }\Omega^c\text{ for a.e. }t,\end{equation} and defining \[\mathbb{K}(t):=H^s_0(\Omega)+g(t)\text{ for a.e. } t\in]0,T[,\]
from \eqref{Equivalence2} with $w=\tilde{w}(t)-\frac{\partial\Theta}{\partial t}(t)$, where $\tilde{w}(t)\in \mathbb{K}(t)$, we obtain, for almost every $t$, \begin{multline}\label{Equivalence3}\int_\Omega b\left(\frac{\partial\Theta}{\partial t}\right)\left(\tilde{w}-\frac{\partial\Theta}{\partial t}\right)+\int_{\mathbb{R}^d} AD^s\Theta\cdot D^s\left(\tilde{w}-\frac{\partial\Theta}{\partial t}\right)+\int_\Omega \lambda\tilde{w}^+-\int_\Omega \lambda\left(\frac{\partial\Theta}{\partial t}\right)^+\\\geq\int_\Omega (\mathfrak{F}(t)+\eta_0)\left(\tilde{w}-\frac{\partial\Theta}{\partial t}\right), \quad\forall \tilde{w}(t)\in\mathbb{K}(t),\end{multline}
which corresponds to the variational inequality formulations of Duvaut and Fr\'emond (see \cite{damlamian1976thesis}, \cite{TarziaThesis}, \cite{TarziaDuvaut1} and \cite{RodriguesVarMethods1994b}). With the same assumptions on $f$, $\tilde{g}$ and $\eta_0$, we can obtain a solution $\Theta$ to \eqref{Equivalence3}, \eqref{Equivalence1} using the Faedo-Galerkin method (refer to \cite{TarziaDuvaut1} or Chapter 3 of \cite{RodriguesVarMethods1994b} for a proof starting from the variational inequality formulation \eqref{Equivalence3}, using the special basis of Appendix \ref{Sect:Eigen}. A similar result can also be obtained using the Rothe method (refer to Section 3.1 of \cite{VisintinStefan}).

Similarly, for the one phase problem we can also obtain an equivalent variational inequality formulation, now of obstacle type. Indeed, governed by $\gamma^\omicron $, the weak temperature solution $\vartheta^\omicron $ obtained in (\ref{FracStefan}$_{1ph}$) is non-negative at all times $t\in[0,T]$. Therefore, its primitive \[\Theta^\omicron (t)=\int_0^t\vartheta^\omicron (\tau)\,d\tau\] is also always non-negative, and satisfies \begin{equation}\tag{\ref{Equivalence1}$^\omicron $}\Theta^\omicron \in H^1(0,T;H^s(\mathbb{R}^d)),\quad\Theta^\omicron (0)=0\text{ and } \Theta^\omicron (t)\geq0,\quad\Theta^\omicron (t)-\int_0^tg(\tau)\,d\tau=0 \text{ in }\Omega^c \text{ for a.e. }t\in]0,T[,\end{equation} and from \eqref{Equivalence0}, denoting $\chi_\omicron\in H(\vartheta^\omicron)$, \begin{equation}\tag{\ref{Equivalence0}$^\omicron $}\int_\Omega b\left(\frac{\partial\Theta^\omicron }{\partial t}(t)\right)w+\int_{\mathbb{R}^d} AD^s\Theta^\omicron (t)\cdot D^sw+\int_\Omega \lambda\chi_\omicron (t)w=\int_\Omega \mathfrak{F}(t)w+\int_\Omega \eta_0w,\quad\text{ for a.e. }t,\forall w\in H^s_0(\Omega).\end{equation} 

Now introduce \[\mathbb{K}^+(t):=\left\{v\in H^s(\mathbb{R}^d):v\geq0\text{ a.e. in }\Omega,v=\int_0^tg(\tau)\,d\tau\text{ in }\Omega^c\right\}, \text{ for a.e. } t\in]0,T[.\] 
Assuming that $\chi_{\{\vartheta^\omicron(t) >0\}}=\chi_{\{\Theta^\omicron(t)>0\}}$ and $\chi_{\{\vartheta^\omicron(t) <0\}}=\chi_{\{\Theta^\omicron(t)<0\}}$ for a.e. $t\in]0,T[$, we can once again make use of the inequality \eqref{signfuncineq} to obtain \[\lambda\chi_\omicron (v-\Theta^\omicron )\leq \lambda(v^+-{\Theta^\omicron}^+)=\lambda(v-\Theta^\omicron )\] when $v(t),\Theta^\omicron (t)\geq0$. Therefore, we can rewrite the equation (\ref{Equivalence0}$^\omicron $) with $w=v-\Theta^\omicron (t)$ for $v\in \mathbb{K}^+(t)$ as a variational inequality to obtain the following evolutionary obstacle-type problem for \, $\Theta^\omicron (t)\in\mathbb{K}^+(t)$: \[\int_\Omega b\left(\frac{\partial\Theta^\omicron }{\partial t}(t)\right)(v-\Theta^\omicron (t))+\int_{\mathbb{R}^d} AD^s\Theta^\omicron (t)\cdot D^s(v-\Theta^\omicron (t))\geq\int_\Omega (\mathfrak{F}(t)+\eta_0-\lambda)(v-\Theta^\omicron (t))\quad\forall v\in\mathbb{K}^+(t).\] 
This corresponds to the nonlocal version of the parabolic variational inequality obtained by Duvaut \cite{Duvaut1973} for the one-phase Stefan problem for the classical case $s=1$. See also \cite{RodriguesStefanRevisited}, \cite{RodriguesVarMethods1994b} or \cite{VisintinStefan}.

\section{Appendix - Dependence of Eigenfunctions of $\mathcal{L}_A^s$ on $0<s\leq1$}\label{Sect:Eigen}

Here we show the continuity of the eigenfunctions of $\mathcal{L}_A^s$ with respect to the parameter $s$, $0<s\leq1$. A similar result on $s\nearrow 1$ can be found in Theorem 1.2 of \cite{brasco2015stability} for the nonlocal $p$-Laplacian and Theorem 3.1 of \cite{BonderGammaConvgEigenNonlocal} for other nonlocal operators.

Recalling the compact embeddings $H^1_0(\Omega)\hookrightarrow H^s_0(\Omega)\hookrightarrow H^\sigma_0(\Omega)\hookrightarrow L^2(\Omega)$ for the bounded open set $\Omega\subset\mathbb{R}^d$, with Lipschitz boundary, where $0<\sigma<s<1$, consider the operator $T^s:L^2(\Omega)\to H^s_0(\Omega)\hookrightarrow L^2(\Omega)$, which depends on $s$, defined by $u^s=T^s(h)\in H^s_0(\Omega)$ corresponding to the homogeneous Dirichlet condition: \begin{equation}\label{A3HomogPb}u^s\in H^s_0(\Omega):\quad\langle\mathcal{L}_A^su^s,v\rangle=\int_{\mathbb{R}^d} AD^su^s\cdot D^sv=\int_\Omega hv, \quad\forall v\in H^s_0(\Omega).\end{equation} Then, by the Poincar\'e inequality, we have \begin{equation}\label{A3}\norm{u^s}_{L^2(\Omega)}^2\leq \frac{C_P}{s}\norm{D^su^s}_{L^2(\mathbb{R}^d)^d}^2\leq \frac{C_P}{sa_*}\langle\mathcal{L}_A^su^s,u^s\rangle\leq\frac{C_P}{sa_*}\int_\Omega hu^s\leq\frac{C_P}{sa_*}\norm{h}_{L^2(\Omega)}\norm{u^s}_{L^2(\Omega)}.\end{equation} Therefore, for $\sigma<s$, \[\norm{T^s}=\sup_{\norm{h}_{L^2(\Omega)}\leq1}\norm{T^s(h)}_{L^2(\Omega)}=\sup_{h\in L^2(\Omega)} \frac{\norm{u^s}_{L^2(\Omega)}}{\norm{h}_{L^2(\Omega)}}\leq\frac{C_P}{sa_*}\leq\frac{C_P}{\sigma a_*}.\]

By the estimate \eqref{A3}, for $\sigma\leq s\to r\leq1$, $u^s$ converges strongly to some $u^*$ in $L^2(\Omega)$. As argued in Section 3.2 of \cite{FracObsRiesz}, $\norm{D^s u^s}_{L^2(\mathbb{R}^d)^d}\leq C$ for some constant $C$ independent of $s$. Therefore, \[D^su^s\xrightharpoonup[s\to r]{}\zeta\quad \text{ in }L^2(\mathbb{R}^d)^d\text{-weak}\] for some $\zeta$.

Now, for all $\Phi\in C_c^\infty(\mathbb{R}^d)^d$, for $s\to r$ \[D^s\cdot\Phi\to D^r\cdot\Phi\quad\text{ in }L^2(\mathbb{R}^d)^d,\] therefore \[\int_{\mathbb{R}^d} D^su^s\cdot\Phi=-\int_{\mathbb{R}^d} u^s(D^s\cdot\Phi)\xrightarrow[s\to r]{}-\int_{\mathbb{R}^d} u^*(D^r\cdot\Phi).\]  But by the a priori estimate on $D^su^s$, \[\left|\int_{\mathbb{R}^d} D^su^s\cdot\Phi\right|\leq C\norm{\Phi}_{L^2(\mathbb{R}^d)^d},\] which implies that \[\left|\int_{\mathbb{R}^d} u^*(D^r\cdot\Phi)\right|\leq C\norm{\Phi}_{L^2(\mathbb{R}^d)^d}\quad \forall\Phi\in C_c^\infty(\mathbb{R}^d)^d.\] This means that $D^ru^*\in L^2(\mathbb{R}^d)^d$, and since $\Omega$ has a Lipschitz boundary, $u^*\in H^r_0(\Omega)$.

Furthermore, since $D^s\cdot \Phi\to D^r \cdot \Phi$ strongly in $L^2(\mathbb{R}^d)^d$ as $s\to r$, so \[\int_{\mathbb{R}^d} D^s(u^s-u^*)\cdot\Phi=-\int_{\mathbb{R}^d}(u^s-u^*)(D^s\cdot\Phi)\to0\quad \forall\Phi\in C_c^\infty(\mathbb{R}^d)^d,\] therefore \[\zeta=w-\lim_{s\to r}D^su^s=D^ru^*\in L^2(\mathbb{R}^d)^d.\]

Taking test functions $\varphi\in C_c^\infty(\Omega)$,
\[\int_{\mathbb{R}^d} AD^ru^*\cdot D^r\varphi=\lim_{s\to r}\int_{\mathbb{R}^d} AD^su^s\cdot D^s\varphi=\lim_{s\to r}\int_\Omega h\varphi=\int_\Omega  h\varphi\quad\forall \varphi\in C_c^\infty(\Omega).\] Extending this by density to all test functions $v\in H^r_0(\Omega)$, by the uniqueness of the solution to the homogeneous Dirichlet boundary problem \eqref{A3HomogPb} with $s=r\leq1$, we have that $u^*=u^r$. Therefore, for every $h\in L^2(\Omega)$, $T^s(h)$ converges to $T^r(h)$ in $L^2(\Omega)$ as $s\to r$.

\begin{theorem}
Let $0<\sigma\leq s, r \leq1$. For the sequence of operators $T^s:L^2(\Omega)\to L^2(\Omega)$ given above, $T^s$ converges to $T^r$ strongly in the operator norm as $s\to r$.
\end{theorem}

\begin{proof}
We first claim that, for each fixed $s$, it is possible to find an $h^s$ in the unit ball of $L^2(\Omega)$ achieving the supremum, i.e. \[\sup_{\norm{h}_{L^2(\Omega)}\leq1}\norm{T^s(h)-T^r(h)}_{L^2(\Omega)}=\norm{T^s(h^s)-T^r(h^s)}_{L^2(\Omega)}.\] Indeed, for any maximizing sequence $\{h_m\}_m$, we can extract a subsequence which converges weakly to some $h^s$ which also belongs to the unit ball of $L^2(\Omega)$. Since the embedding from $L^2(\Omega)$ into $H^{-\sigma}(\Omega)\subset H^{-s}(\Omega)\cap H^{-r}(\Omega)$ is compact, and since $T^s$ and $T^r$ can also be considered continuous operators from $H^{-s}(\Omega)$ into $H^s_0(\Omega)$ and from $H^{-r}(\Omega)$ into $H^r_0(\Omega)$, respectively, both operators are also  completely-continuous operators in $L^2(\Omega)$, and so taking $m$ to infinity we have the conclusion.

Having obtained the sequence $\{h^s\}_s$, since they are the weak limits of a uniformly bounded sequences, there exists $h$ in the unit ball of $L^2(\Omega)$ such that $h^s$ converge weakly in $L^2(\Omega)$ and strongly in $H^{-\sigma}(\Omega)$ to $h$. Then, by Lemma \ref{sContDepStefan}, for $\sigma\leq s$, we have $\norm{u}_{H^\sigma_0(\Omega)}\leq c_\sigma\norm{u}_{H^s_0(\Omega)}$ for $u\in H^s_0(\Omega)$ and consequently \[\norm{h}_{H^{-s}(\Omega)}=\sup_{u\in H^s_0(\Omega)}\frac{\langle h,u\rangle}{\norm{u}_{H^s_0(\Omega)}}\leq c_\sigma\norm{h}_{H^{-\sigma}(\Omega)}.\] As in \eqref{A3}, if $u=T^s(f)$ with $f\in H^{-s}(\Omega)$, we obtain \[a_*\norm{u}_{H^s_0(\Omega)}^2=a_*\norm{D^su}_{L^2(\mathbb{R}^d)^d}^2\leq\langle\mathcal{L}_A^su,u\rangle=\int_\Omega fu\leq \norm{f}_{H^{-s}(\Omega)}\norm{u}_{H^s_0(\Omega)},\quad\forall f\in H^{-s}(\Omega),\] and then 
\[\norm{T^s}_{s} =\sup_{f\in H^{-s}(\Omega)}  \frac{\norm{u}_{H^s_0(\Omega)}}{\norm{f}_{H^{-s}(\Omega)}} \leq \frac{1}{a_*}\] for the operator norm $\norm{\cdot}_s$ as an operator from $H^{-s}(\Omega)$ to $H^s(\Omega)$. Therefore, it follows that
\[\norm{T^s}_{\sigma}= \sup_{f\in H^{-\sigma}(\Omega)} \frac{\norm{T^s(f)}_{H^\sigma_0(\Omega)}}{\norm{f}_{H^{-\sigma}(\Omega)}}\leq c_\sigma^2\sup_{f\in H^{-s}(\Omega)} \frac{\norm{T^s(f)}_{H^s_0(\Omega)}}{\norm{f}_{H^{-s}(\Omega)}}=c_\sigma^2\norm{T^s}_{s}\leq \frac{c_\sigma^2}{a_*}.\] Similarly, we have \[\norm{T^r}_\sigma\leq\frac{c_\sigma^2}{a_*}.\]  
Since $T^s(h)$ converges to $T^r(h)$ in $L^2(\Omega)$ for every $h\in L^2(\Omega)$, for any $\epsilon>0$, we can pick a $\delta>0$ such that, for $|s-r|\leq\delta$, we have  \[\norm{h^s-h}_{H^{-\sigma}(\Omega)}\leq\frac{\epsilon a_*}{4c_\sigma^2}\quad\text{ and }\norm{T^s(h)-T^r(h)}_{L^2(\Omega)}\leq\frac{\epsilon}{2}.\] Therefore, \begin{align*}\sup_{\norm{f}_{L^2(\Omega)}\leq1}\norm{T^s(f)-T^r(f)}_{L^2(\Omega)}&=\norm{T^s(h^s)-T^r(h^s)}_{L^2(\Omega)}\\&\leq\norm{T^s(h)-T^r(h)}_{L^2(\Omega)}+\norm{T^s(h^s-h)-T^r(h^s-h)}_{L^2(\Omega)}\\&\leq\frac{\epsilon}{2}+\left(\norm{T^s}_\sigma+\norm{T^r}_\sigma\right)\norm{h^s-h}_{H^{-\sigma}(\Omega)}\\&\leq\frac{\epsilon}{2}+\frac{2c_\sigma^2}{a_*}\frac{\epsilon a_*}{4c_\sigma^2}=\epsilon.\end{align*}

\end{proof}

As a corollary, by Theorem 2.3.1 of \cite{HenrotBook}, we have

\begin{corollary}
For the operators $T^s$, $T^r$ as given in the previous theorem, let $\lambda_k^s=\lambda_k^s(T^s)$ and $\lambda_k^r=\lambda_k^r(T^r)$be the $k$-th eigenvalues of $T^s$ and of  $T^r$ respectively for $s$ and for $r$, $0<\sigma \leq s, r\leq1$. Then, \[|\lambda_k^s-\lambda_k^r|\leq\norm{T^s-T^r}:=\sup_{\norm{f}_{L^2(\Omega)}\leq1}\norm{(T^s-T^r)(f)}.\] In particular, the map $[\sigma,1]\ni s\mapsto \lambda_k^s\in (0,\infty)$ is continuous.
\end{corollary}

For each eigenvalue $\lambda_k^s$, the associated eigenvector $h^s_k$ of $T^s$ such that $T^s(h^s_k)=\lambda_k^s h_k^s$. Setting $u^s_k:=T^s(h^s_k)$, we have $u_k^s=T^s(h^s_k)=\lambda_k^s h_k^s=\lambda_k^s \mathcal{L}_A^su_k^s$, so $1/\lambda_k^s$ is the eigenvalue of $\mathcal{L}_A^s$ with associated eigenvector $u_k^s$.

\begin{corollary}
Let $u_k^s$ be the corresponding eigenfunctions of $1/\lambda_k^s$ for the operator $\mathcal{L}_A^s$ for $s\in[\sigma,r]$, $0<\sigma< r\leq1$. Then, the maps  $[\sigma,1]\ni s\mapsto u_k^s\in L^2(\Omega)$ and $]\sigma,1]\ni r\mapsto u_k^r\in H^\sigma_0(\Omega)$ are also continuous.
\end{corollary}

\begin{proof}
Since $\lambda_k^s$ converges, so does $1/\lambda_k^s$. Therefore, \[a_*\norm{D^su_k^s}_{L^2(\mathbb{R}^d)^d}^2\leq\langle\mathcal{L}_A^su_k^s,u_k^s\rangle=\frac{1}{\lambda_k^s}\norm{u_k^s}_{L^2(\Omega)}^2.\] Normalizing by $\norm{u_k^s}_{L^2(\Omega)}=1$, the convergence of the eigenvalues gives \[a_*\norm{D^su_k^s}_{L^2(\mathbb{R}^d)^d}^2\leq\left(\frac{1}{\lambda_k^s}-\frac{1}{\lambda_k^r}\right)+\frac{1}{\lambda_k^r}\leq1+\frac{1}{\lambda_k^r}\] for $|r-s|$ sufficiently small and for $r\leq1$ and $k$ fixed. This means that the $H^s_0(\Omega)$ norm of $u_k^s$ is bounded, so by compactness, there exists a sequence $\{s_n\}_{n\in\mathbb{N}}$ with $s_n\to r$ such that the corresponding sequence of eigenfunctions $\{u_k^{s_n}\}_{n\in\mathbb{N}}$ converges weakly in $H^\sigma_0(\Omega)$ and strongly in $L^2(\Omega)$ to some $u^*_k$ for each $k$. This $u^*_k$ corresponds to a $h^*_k=\frac{1}{\lambda_k^r}u^*_k$ which is the limit of $h^s_k$, where $h^s_k$ satisfies $T^s(h^s_k)=\lambda^s_kh^s_k$. Since $\lambda^s_k\to\lambda^r_k$, $h^s_k=\frac{1}{\lambda_k^s}u^s_k$ converges to $h^*_k=\frac{1}{\lambda_k^r}u^*_k$ strongly in $L^2(\Omega)$ as $s\to r$, and by the convergence of the operator norm $T^s\to T^r$, \[T^s(h^s_k)\to T^r(h^*_k)\quad\text{ and }\quad\lambda^s_k\to\lambda^r_k \quad\text{ as }s\to r.\] Now, by the definition, the image of $T^r$ lies in $H^r_0(\Omega)$, so $u^*_k=\lambda_k^rh^*_k=T^r(h^*_k)\in H^r_0(\Omega)$. Consequently, $h^*_k=h^r_k$, so $u^*_k=u^r_k$. Therefore, for every fixed $k$ and $r$, $u^s_k$ converges strongly to $u^r_k$ in $L^2(\Omega)$ as $s\to r$, with $\norm{u_k^r}_{L^2(\Omega)}=1$, which yields the continuity of the map $[\sigma,1]\ni s\mapsto u_k^s\in L^2(\Omega)$. Since $r>\sigma$, by the compactness of the inclusion $H^{\sigma'}_0(\Omega)\hookrightarrow H^\sigma_0(\Omega)$ for all $\sigma'>\sigma$, we also have the continuity of the map $]\sigma,1]\ni r\mapsto u_k^r\in H^\sigma_0(\Omega)$.

\end{proof}

\section*{}

\noindent \textbf{Acknowledgements.} 
C. Lo acknowledges the FCT PhD fellowship in the framework of the LisMath doctoral programme at the University of Lisbon. The research of J. F. Rodrigues was partially done under the framework of the Project PTDC/MATPUR/28686/2017 at CMAFcIO/ULisboa. We would also like to thank the referees for their insightful comments.

\bibliographystyle{plain}
\bibliography{ref.bib}

\end{document}